\documentclass{amsart}

\usepackage{xcolor}
\usepackage{amsmath,amsthm, amssymb,stmaryrd,mathrsfs, scalefnt}
\usepackage{graphicx,enumerate}
\usepackage[a4paper, total={6.5in, 8.5in}, bottom = .8in, top = 1in]{geometry}

\usepackage{scalerel}
\makeatletter
\DeclareMathOperator*{\bigE}{\scaleobj{1.2}{\exists}}
\DeclareMathOperator*{\bigM}{\bigcurlywedge}
\makeatother
\newcommand{\iex}[1]{\underset{#1}{\bigE}}
\newcommand{\siex}{\bigE}

\usepackage{parskip} % supresses indentation of first lines of paragraphs (just this line, no explicit commands)

\newcommand{\nocomma}{}
\newcommand{\tmmathbf}[1]{\ensuremath{\boldsymbol{#1}}}
\newcommand{\tmop}[1]{\ensuremath{\operatorname{#1}}}
\newcommand{\tmrsub}[1]{\ensuremath{_{\textrm{#1}}}}
\newcommand{\tmrsup}[1]{\textsuperscript{#1}}
\newcommand{\tmtextbf}[1]{\text{{\bfseries{#1}}}}

\newcommand{\triplesim}{\asymp}%{\Bumpeq}
\newcommand{\tmtextit}[1]{\text{{\itshape{#1}}}}
\newcommand{\longrightarrowlim}{\mathop{\longrightarrow}\limits}
\def\A{\mathcal{A}}
\def\sepa{\mathscr{S}}
\def\theasm{\mathbf{Asm}}

\newcommand{\restr}[2]{#1_{\mid_{#2}}}
\newcommand{\asm}[1]{\theasm_{#1}}
\newcommand{\asma}{\asm{\A}}
\newcommand{\mass}{\asm{M}}

\newcommand{\intpr}[1]{(#1)\tmrsup{\ensuremath{\mathcal{A}}}}
\newcommand{\lam}[2]{\ensuremath{\lambda #1 . #2}}
\newcommand{\ilam}[2]{{\intpr{{\lam{#1}{#2}}}}}
\newcommand{\lamm}[1]{\ensuremath{\tmmathbf{\lambda} #1}}

\newcommand{\exi}[1]{\mathsf{E}_{#1}}
\newcommand{\eexi}[1]{\mathsf{E}^{\approx}_{#1}}
\newcommand{\exiplus}[1]{\mathsf{E}^{+}_{#1}}

\newcommand{\car}[1]{{\underline{#1}}}
\newcommand{\lpair}[2]{\ensuremath{\tmmathbf{\left\langle #1, #2
\right\rangle}}}
\newcommand{\lpizero}{\ensuremath{\tmmathbf{{\pi}}}\tmrsub{0}}
\newcommand{\lpione}{\ensuremath{\tmmathbf{{\pi}}}\tmrsub{1}}
\newcommand{\lpitwo}{\ensuremath{\tmmathbf{{\pi}}}\tmrsub{2}}

\def\trip{\mathbb{P}}
\newcommand{\Rf}{\ensuremath{\mathfrak{R}_f}}

\newcommand{\tripalg}[1]{\trip \left[ #1 \right]}

\newcommand{\fibralg}[2]{\tripalg{#1}_{#2}}

\newcommand{\fibr}[1]{\fibralg{\A}{#1}}
\newcommand{\sfibr}[1]{\fibralg{\sepa}{#1}}
\newcommand{\mfibr}[1]{\fibralg{M}{#1}}

\newcommand{\Equ}[1]{\tmop{Equ}\tmrsub{#1}\tmrsup{+}}
\newcommand{\equ}{\ensuremath{\approx}}
\newcommand{\ngh}[1]{#1\tmrsup{+}}
\newcommand{\dum}{\ensuremath{\diamond}}
\newcommand{\vld}[1]{{\Vdash} #1}
\newcommand{\letin}[2]{\tmtextbf{let} \; #1 \; \tmtextbf{=} \; #2 \; \tmtextbf{in} \;}

\newcommand{\pows}[1]{\ensuremath{\mathfrak{P}}(#1)}
\newcommand{\ppows}[1]{\ensuremath{\mathfrak{P}_{\bullet}\left(#1\right)}}

\newcommand{\arobj}[3]{\ensuremath{\scalefont{.75} \left[\begin{array}{l}
  #2\\
  \downarrow #1\\
  #3
\end{array}\right]}}

\newcommand{\pr}[2]{#1, #2}
\newcommand{\y}{\ensuremath{\tmmathbf{y}}}
\newcommand{\tops}{\ensuremath{\tmmathbf{\tmop{Top}}}}
\newcommand{\Reg}{\ensuremath{\tmmathbf{\tmop{Reg}}}}
\newcommand{\Lex}{\ensuremath{\tmmathbf{\tmop{Lex}}}}
\newcommand{\Ex}{\ensuremath{\tmmathbf{\tmop{Ex}}}}
\newcommand{\ens}{\ensuremath{\tmmathbf{\tmop{Set}}}}
\newcommand{\cat}{\ensuremath{\tmmathbf{\tmop{Cat}}}}
\newcommand{\ab}{\ensuremath{\tmmathbf{\tmop{Ab}}}}
\newcommand{\qgrpd}[1]{\tmtextbf{PGrpd}(#1)}
\newcommand{\eff}{\ensuremath{\mathcal{E}\!f\!f}}
\newcommand{\omegaset}{\ensuremath{\omega \ens}}
\newcommand{\pass}{\ensuremath{\mathbf{P}}\ensuremath{\mathbf{A}}\ensuremath{\mathbf{s}}\ensuremath{\mathbf{s}}}
\newcommand{\eset}[1]{(#1,\: {\approx})}
\newcommand{\eeq}[2]{\mid \! #1 \: {\approx} \: #2 \! \mid}
\newcommand{\ieeq}[3]{\ensuremath{\left| #1 \: \approx_{#3} \: #2 \right|}}
\newcommand{\eeqq}[2]{\mid#1 \: {\triplesim} \: #2\mid}

\newcommand{\sym}[1]{\mathsf{S}\tmrsub{#1}}
\newcommand{\trans}[1]{\mathsf{T}\tmrsub{#1}}
\newcommand{\ext}[1]{\mathsf{Ext}\tmrsub{#1}}
\newcommand{\str}[1]{\mathsf{Str}\tmrsub{#1}}
\newcommand{\sv}[1]{\mathsf{Sv}\tmrsub{#1}}
\newcommand{\tot}[1]{\mathsf{Tot}\tmrsub{#1}}
\newcommand{\ide}{\ensuremath{\tmmathbf{\mathtt{I}}}}

\newcommand{\reglex}[1]{\ensuremath{#1_{\tmop{reg} / \tmop{lex}}}}
\newcommand{\exreg}[1]{\ensuremath{#1_{\tmop{ex} / \tmop{reg}}}}
\newcommand{\exlex}[1]{\ensuremath{#1_{\tmop{ex} / \tmop{lex}}}}
\newcommand{\ensa}{{\ens}[\ensuremath{\mathcal{A}}]}

\newcommand{\tmem}[1]{{\em #1\/}}

\newcommand{\assign}{:=}
\newcommand{\nin}{\not\in}
\newcommand{\tmdummy}{$\mbox{}$}
\newenvironment{enumeratenumeric}{\begin{enumerate}[1.] }{\end{enumerate}}
\newenvironment{enumerateroman}{\begin{enumerate}[i.] }{\end{enumerate}}
\newenvironment{itemizeminus}{\begin{itemize} }{\end{itemize}}

\newtheorem{theorem}{Theorem}[section]
\newtheorem{lemma}{Lemma}[section]
\newtheorem{proposition}{Proposition}[section]
\newtheorem{corollary}{Corollary}[section]

\theoremstyle{definition}
\newtheorem{definition}{Definition}[section]

\newtheorem{example}{Example}[section]
\theoremstyle{remark}
\newtheorem{remark}{Remark}[section]
\newtheorem*{notation}{Notation}

\def\IMERL{Instituto de Matem{\'a}tica y Estad{\'i}stica Rafael Laguardia
  (IMERL), Facultad de Ingenier{\'i}a, Universidad de la Rep{\'u}blica,
  Julio Herrera y Reissig 565, C.P. 11300 Montevideo, Uruguay}

\def\LAMA{Laboratoire de Math{\'e}matiques (LAMA),
  Universit{\'e} Savoie Mont Blanc,
  B{\^a}timent Le Chablais, Campus scientifique,
  73376 Le Bourget du Lac, France}

\begin{document}

{{\title{Completions of Implicative
Assemblies}

\author[A. Miquel]{Alexandre Miquel}
\address{\IMERL}
\author[K. Worytkiewicz]{Krzysztof Worytkiewicz}
\address{\LAMA}
\date{}

\maketitle}}

\begin{abstract}
  We continue our work on implicative assemblies by investigating under which
  circumstances a subset $M \subseteq \sepa$ gives rise to a lex full
  subcategory $\mass$ of the quasitopos
  $\asma$ of all assemblies such that
  $\reglex{(\mass)} \simeq
  \asma$. We establish a
  characterisation. Furthermore, this latter is relevant to the study of
  $\mass$'s ex/lex-completion.
\end{abstract}

%{\tableofcontents}
\tableofcontents
\clearpage

\section{Introduction}
\label{sec:intro}

In this note we investigate under which circumstances a subset $M \subseteq
\sepa$ gives rise to a lex full subcategory $\mass$
of the quasitopos $\asma$ of all
assemblies such that $\reglex{(\mass)} \simeq
\asma$. It is well-known that we have
such a situation in the effective topos $\eff$
{\cite{hyland1980tripos,hyland1982effective}} where there is the subcategory
$\pass \subseteq \omegaset \subseteq \eff$ whose objects are known as
{\tmem{partitioned assemblies}} (their existence predicates are singletons
\cite{robinson1990colimit,rosolini2019elementary}). We establish a
characterisation for the implicative case and the case of partitioned
assemblies turns out to be quite easy going. We next observe that
$\exreg{(\asma)}$ is always a topos,
which is a rather immediate consequence of the material in
{\cite{menni2000exact}}. This observation implies that
$\exlex{(\mass)}$ is a topos as well. We then
construct a functor
\[ \mathcal{K}: \exlex{(\mass)} \rightarrow \ens
   [\mathcal{A}] \]
the latter being the topos constructed from an implicative tripos
{\cite{miquel2020implicative,Castro:2023aa}} by means of the tripos-to-topos
construction. It turns out that $\mathcal{K}$ may not always be an equivalence, although it is always the case when $\mathcal{A}$ is compatible with
joins. There are undoubtly connections to the material in
{\cite{maietti2015unifying}}, we only (very) recently became aware of the
latter work. Then again our focus is quite different as we specifically
investigate categorical aspects of implicative algebras. The reader is
referred to {\cite{miquel2020implicative,Castro:2023aa}} for basic material
about implicative algebras and implicative assemblies.

\section{Regular Completions}
\label{sec:regular}

\subsection{Memjog on Regular Categories}
{\tmdummy}\\ \\
The notion of regular category axiomatises image factorisations, as a matter
of fact an abelian category is an additive regular category. It is thus not
really surprising that regular categories have been studied for some time
{\cite{barr1971exact}}.

\begin{definition}
  \leavevmode %{\tmdummy}

  \begin{enumeratenumeric}
    \item A {\tmem{regular}} epi is an epi which is a coequaliser (in any
    category).

    \item A {\tmem{lex category}} is a category with finite limits.
  \end{enumeratenumeric}
\end{definition}

\begin{remark}
  In a lex category, a regular epi is a coequaliser if its kernel pair.
\end{remark}

\begin{definition}
  A lex category $\mathbb{D}$ is regular if
  \begin{enumerateroman}
    \item it has coequalisers of kernel pairs;
    \item regular epis are stable under pullback.
  \end{enumerateroman}
\end{definition}

Despite notorious non-examples like $\tops$ (always troubles with this one..)
or $\cat$, regular categories abound in nature: $\ens$,
$\mathbf{G}\mathbf{r}\mathbf{p}$ (all groups), $\ab$, any LCCC with
coequalisers (thus any quasitopos), any topos etc etc. As mentioned, any
abelian category is regular as well. This is the reason why in the context of
regular categories we tend to call {\tmem{exact sequence}} a coequaliser
diagram of a kernel pair. For, the datum

\begin{center}
  \raisebox{-0.5\height}{\includegraphics[width=14.8909550045914cm,height=0.793585202676112cm]{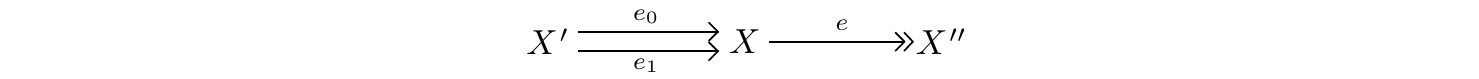}}
\end{center}

(with $(e_0, e_1)$ the kernel pair of $e$) in an abelian category yields the
short exact sequence

\begin{center}
  \raisebox{-0.5\height}{\includegraphics[width=14.8909550045914cm,height=1.0910730683458cm]{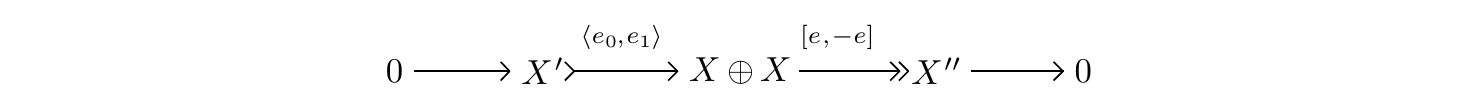}}
\end{center}

and vice versa. An equivalent definition of a regular category is then as
above but with condition $i \nocomma i.$ replaced by requiring that exact
sequences are stable under pullback.

Given any morphism $f : X \rightarrow Y$ in a regular category we can
construct the diagram

\begin{center}
  \raisebox{-0.5\height}{\includegraphics[width=14.8909550045914cm,height=1.98350386986751cm]{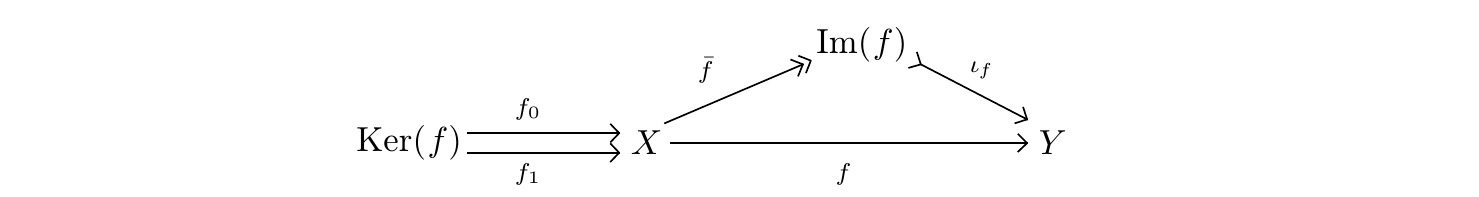}}
\end{center}

with $\tmop{Im} (f) \cong X / \tmop{Ker} (f)$, $\bar{f}$ the coequaliser of
$f$'s kernel pair $(f_0, f_1)$ and $\iota_f$ given by universal property (it
is an exercise in diagram chasing to prove that $\iota_f$ is mono). In an
abelian category it is just the first isomorphism theorem.

\begin{notation}
  Given a morphism $f$ in a regular category, we shall write $f = \iota_f
  \circ \bar{f}$ for its factorisation with $\iota_f$ mono and $\bar{f}$
  regular epi.
\end{notation}

Regular epis and monos form an orthogonal (thus functorial) factorisation
system $(\mathcal{E}_{\tmop{reg}}, \mathcal{M})$. Another equivalent
definition of a regular category is as a lex category with pullback-stable
image factorisations.

\begin{example}
  $\mathbf{A}\mathbf{s}\mathbf{m}_{\mathcal{A}}$ is a quasitopos, hence
  regular. It is nonetheless handy to know how to calculate image
  factorisations. Let $f : X \rightarrow Y$ be a morphism in
  $\mathbf{A}\mathbf{s}\mathbf{m}_{\mathcal{A}}$. Its kernel pair is given by
  the pullback square

  \begin{center}
    \raisebox{-0.5\height}{\includegraphics[width=14.8909550045914cm,height=2.87590187590188cm]{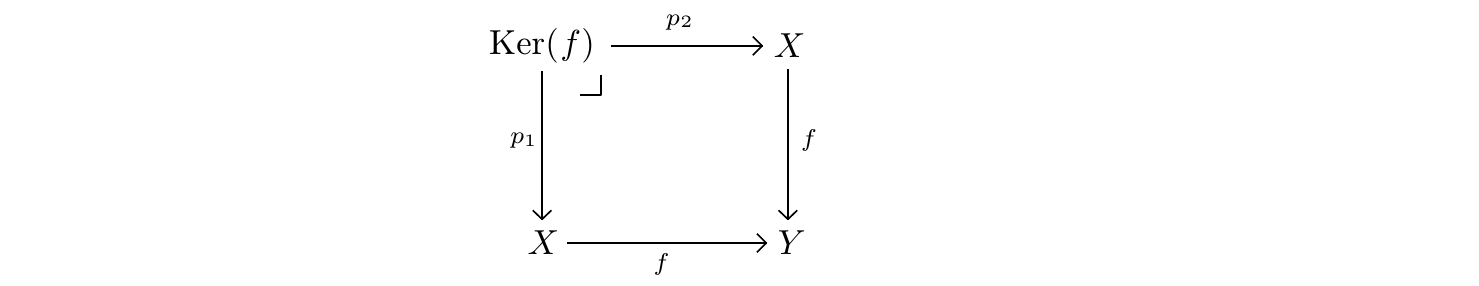}}
  \end{center}

  Recall that a pullback square in
  $\mathbf{A}\mathbf{s}\mathbf{m}_{\mathcal{A}}$ is given on carriers by the
  forgetful functor $\Gamma : \mathbf{A}\mathbf{s}\mathbf{m}_{\mathcal{A}}
  \rightarrow \ens$, with existence and projections inherited from the product
  in $\mathbf{A}\mathbf{s}\mathbf{m}_{\mathcal{A}}$, so
  $\exi{\tmop{Ker} (f)} (x, x') = x \sqcap x'$
  while $p_1$ and $p_2$ are tracked by
  $\left(\lam{z}{z \left( \lam{x \nocomma y}{x} \right)} \right)^{\mathcal{A}}$
  and
  $\left( \lam{z}{z \left( \lam{x \nocomma y}{\nocomma y} \right)}
  \right)^{\mathcal{A}}$, respectively. Let $[x] \assign f^{- 1} (f (x))$ be
  the inverse image of $f (x)$. The coequaliser diagram

  \begin{center}
    \raisebox{-0.5\height}{\includegraphics[width=14.8909550045914cm,height=1.48770169224715cm]{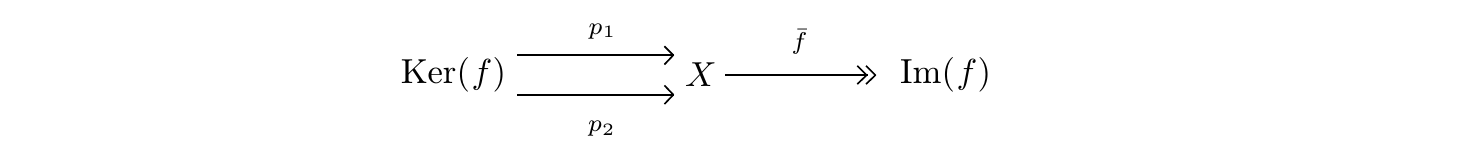}}
  \end{center}

  is given by
  \begin{enumerateroman}
    \item $\car{\tmop{Im} (f)} = X / \sim$ where $\sim$ is the equivalence
    relation generated by $f (x) \sim f (x')$, so \ $\car{\tmop{Im} (f)}$ is
    the set $\{ [x] |x \in X \}$. The existence predicate is
    $\exi{\tmop{Im} (f)} ([x]) = \iex{x' \in [x]} \exi{X} (x')$;

    \item $\bar{f} : x \mapsto [x]$ is the canonical projection, tracked by
    $\left( \lam{x \nocomma z}{z \nocomma x} \right)^{\mathcal{A}}$.
  \end{enumerateroman}
  The image factorisation

  \begin{center}
    \raisebox{-0.5\height}{\includegraphics[width=14.8909550045914cm,height=1.8843303161485cm]{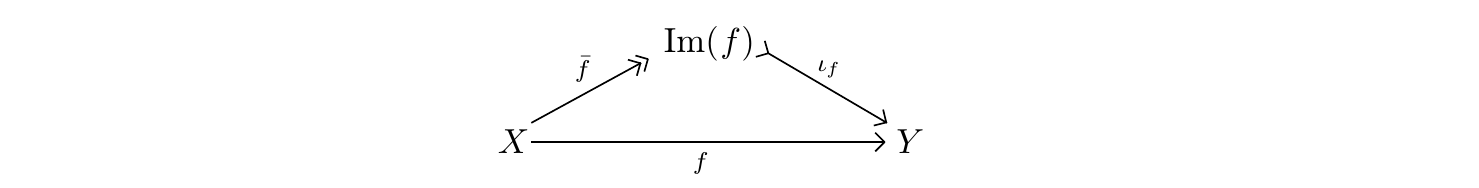}}
  \end{center}

  is in turn given by universal property of $\tmop{Im} (f)$. Specifically, we
  have $\iota_f ([x]_f) = f (x)$ independently of the choice of represantant
  $x$, tracked by $\ilam{z}{z \tau}$ with $\tau$ some tracker of $f$. It can be
  seen without invoking category theory that $\iota_f$ is injective as a map,
  thus a mono in $\mathbf{A}\mathbf{s}\mathbf{m}_{\mathcal{A}}$.
\end{example}

\subsection{The Universal Property of a Regular Completion}
\label{subsec:reg-universal}

\begin{definition} \leavevmode
  \begin{enumeratenumeric}
    \item A {\tmem{lex functor}} is a functor commuting with finite limits.
    \item Let $\mathbb{D}$ and $\mathbb{D}'$ be regular categories. An
    {\tmem{exact functor}} $R : \mathbb{D} \rightarrow \mathbb{D}'$ is a lex
    functor which preserves regular epis.
  \end{enumeratenumeric}
\end{definition}

Equivalently, a lex functor preserving exact sequences is exact (whence the
name). An exact functor obviously preserves image factorisations.

\begin{notation}
  We shall write
  \begin{itemizeminus}
    \item {\Lex} for the 2-category of lex categories, lex funtors and their
    natural transformations;
    \item {\Reg} for the 2-category of regular categories, exact funtors and
    their natural transformations;
    \item $| - | : \Reg \rightarrow \Lex$ for the forgetful $2$-functor.
  \end{itemizeminus}
\end{notation}

It is part of the lore that $| - |$ has a left biadjoint $(-)_{\tmop{reg} /
\tmop{lex}}$. Given a lex category $\mathbb{C}$ we thus have a lex functor $\y
: \mathbb{C} \rightarrow \mathbb{C}_{\tmop{reg} / \tmop{lex}}$ such that for
any lex functor $L : \mathbb{C} \rightarrow \mathbb{D}$ with $\mathbb{D}$ a
regular category there is a unique exact functor $\mathcal{U}:
\mathbb{C}_{\tmop{reg} / \tmop{lex}} \rightarrow \mathbb{D}$ such that the
triangle

\begin{center}
  \raisebox{-0.5\height}{\includegraphics[width=14.8909550045914cm,height=3.09753377935196cm]{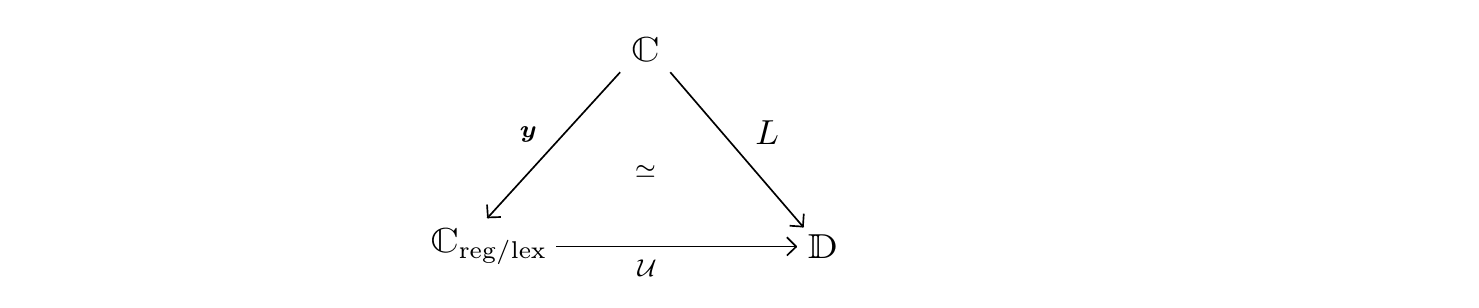}}
\end{center}

commutes (up-to equivalence). The objects of $\mathbb{C}_{\tmop{reg} /
\tmop{lex}}$ are morphisms of $\mathbb{C}$. A morphism
\[ [l] : \arobj{f}{X}{Y} \rightarrow \arobj{f'}{X'}{Y'} \]
in $\mathbb{C}_{\tmop{reg}}$ is an equivalence class $[l]$ of morphisms
$l : X \rightarrow X'$ such that $f' \circ l$ coequalises $f$'s kernel pair, with
$[l_1] = [l_2]$ if $f' \circ l_1 = f' \circ l_2$.

The insertion functor
$\y : \mathbb{C} \rightarrow \mathbb{C}_{\tmop{reg}}  \tmop{lex}$ is given by
$X \mapsto \arobj{\tmop{id}}{X}{X}$

Given a morphism
\[ [f] : \arobj{\tmop{id}}{X}{X} \rightarrow \arobj{\tmop{id}}{Y}{Y} \]
in the image of $\y$, it is obvious that $f$ is the only member of $[f]$ and
easy to see that
\begin{eqnarray*}
  \tmop{Im} ([f]) & = & \arobj{f}{X}{Y}
\end{eqnarray*}
(where $f$ is seen as an object of $\reglex{\mathbb{C}}$). It would be fair to
say that this is the idea behind the construction as the latter freely adds
images.

A frequent use-case is when $\mathbb{D} \in \Reg$ and $\mathbb{C} \subset
\mathbb{D}$ is a lex subcategory. The canonical functor $\mathcal{U}$ making
the diagram

\begin{center}
  \raisebox{-0.5\height}{\includegraphics[width=14.8909550045914cm,height=2.38013249376886cm]{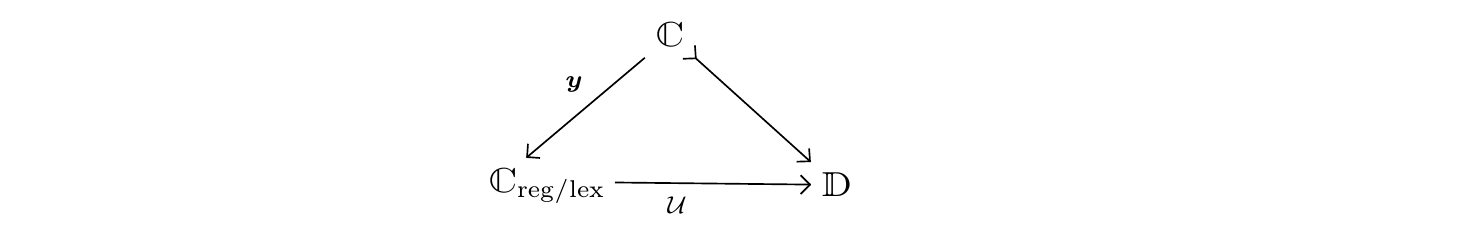}}
\end{center}

commute is particularly easy to describe in this situation:
\begin{eqnarray*}
  \mathcal{U}: \mathbb{C}_{\tmop{reg} / \tmop{lex}} & \longrightarrow &
  \mathbb{D}\\
  \arobj{f}{X}{Y} & \mapsto & \tmop{Im} (f)
\end{eqnarray*}
with action on morphisms given by universal property as follows. Assume a
morphism
\[ [l] : \arobj{f}{X}{Y} \rightarrow \arobj{f'}{X'}{Y'} \]
We then have the diagram

\begin{center}
  \raisebox{-0.499997244131378\height}{\includegraphics[width=14.8909550045914cm,height=2.97505903187721cm]{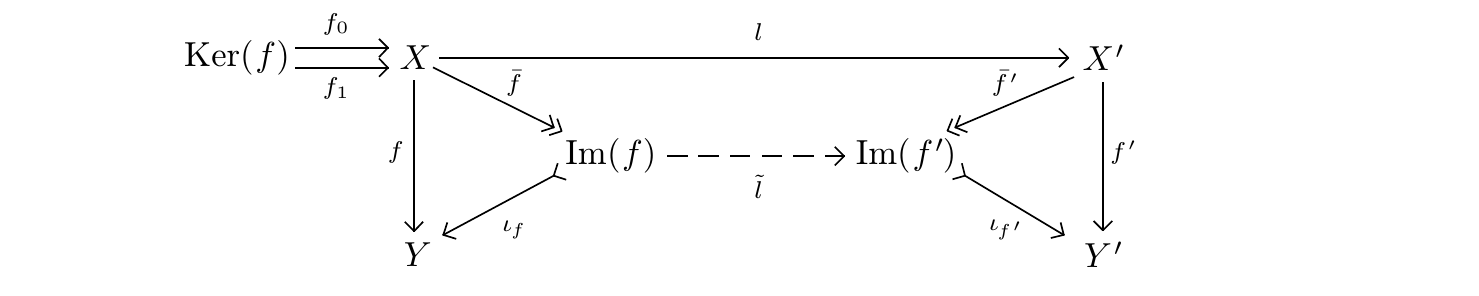}}
\end{center}

in $\mathbb{D}$. As $f' \circ l$ coequalises the kernel pair $(f_0, f_1)$ of
$f$ and $\iota_{f'}$ is mono, $\overline{f'}$ coequalises $(f_0, f_1)$ as well
so $\tilde{l}$ is given by universal property of $\bar{f}$. It is easy to see
that the assigment $[l] \mapsto \tilde{l}$ is well defined and functorial.
Moreover, $\mathcal{U}$ is always faithful in this situation. Finally, we have
$\mathbb{C}_{\tmop{reg} / \tmop{lex}} \simeq \mathbb{D}$ if and only if
$\mathcal{U}$ is (part of) an equivalence of categories.

\begin{definition}
  $\mathbb{D}$ is a {\tmem{regular completion}} of $\mathbb{C}$ if
  $\mathcal{U}$ is an equivalence of categories.
\end{definition}

In practice we thus need to check if $\mathcal{U}$ is essentially surjective
and full. It is well-known that if $\mathbb{D}$ is a {\tmem{regular
completion}} of $\mathbb{C}$, closing $\mathbb{C}$ under isomorphisms yields
$\mathbb{D}$'s subcategory of projective objects.

\section{Some Implicative Remarks}
\label{sec:implicative-remarks}

\subsection{Assemblies as Indexed Families.}

\begin{notation}
  Let $\mathcal{A}$ be an implicative algebra.
  \begin{enumeratenumeric}
    \item $\mathbb{P} [\mathcal{A}]$ stands for the implicative tripos over
    $\mathcal{A}$ and $\fibr{X}$ for its fibre over $X$.

    \item Let $Q \subseteq \sepa$ be a subset of $\sepa$. We shall
    write $\mathbb{P} [Q]$ for the full subcategory of $\mathbb{P}
    [\mathcal{A}]$ of families of elements of $Q$.

    \item $\mathbb{P} [Q]_X \assign \fibr{X} \cap \mathbb{P} [Q]$.
  \end{enumeratenumeric}
\end{notation}

\begin{definition}
  We call $\mathbb{P} [Q]_X$ {\tmem{the $Q$-fibre}} over $X$.
\end{definition}

\begin{proposition}
  \label{prop:tracked}The assignment
  \begin{eqnarray*}
    \Xi : \asma & \longrightarrow &
    \mathbb{P} [\sepa]\\
    X & \mapsto & \exi{X}\\
    {\scriptstyle f} \downarrow  &  & \downarrow {\scriptstyle \left( f, \exi{X} \sqsubseteq_X f^{\ast}
    \exi{Y} \right)}\\
    Y & \mapsto & \exi{Y}
  \end{eqnarray*}
  is an isomorphism of categories with inverse
  \begin{eqnarray*}
    \Xi^{- 1} : \mathbb{P} [\sepa] & \longrightarrow &
    \asma\\
    (u_x)_{x \in X} & \mapsto & (X, u)\\
    {\scriptstyle (f, u \sqsubseteq_X f^{\ast} v)} \downarrow &  & \text{\quad} \downarrow
    {\scriptstyle f}\\
    (v_y)_{y \in Y} & \mapsto & (Y, v)
  \end{eqnarray*}
\end{proposition}

\begin{proof}
  $\Xi_0$ and $\Xi^{- 1}_0$ are obviously inverse bijections (on classes).
  Furthermore, the maps
  \begin{eqnarray*}
    \Xi_{X, Y} : \asma (X, Y) &
    \longrightarrow & \mathbb{P} [\sepa] \left( \exi{X}, \exi{Y}
    \right)\\
    f & \mapsto & \left( f, \exi{X} \sqsubseteq_X f^{\ast} \exi{Y} \right)
  \end{eqnarray*}
  and
  \begin{eqnarray*}
    \Xi^{- 1}_{u, v} : \mathbb{P} [\sepa] ((u_x)_{x \in X}, (v_y)_{y \in
    Y}) & \longrightarrow & \asma (X,
    Y)\\
    (f, u \sqsubseteq_X f^{\ast} v) & \mapsto & f
  \end{eqnarray*}
  are inverse bijections since
  \begin{eqnarray*}
    \text{$f$ is tracked} & \Longleftrightarrow & \bigM_{x \in X}
    \left( \exi{X} (x) \rightarrow \exi{Y} (f (x)) \right) \in \sepa\\
    & \Longleftrightarrow & \bigM \left( \exi{X} \rightarrow_X
    f^{\ast}  \exi{Y} \right) \in \sepa\\
    & \Longleftrightarrow & \exi{X} \sqsubseteq_X f^{\ast} \exi{Y}
  \end{eqnarray*}

\end{proof}

The ismorphism above is obviously rather shallow. Still, it is quite practical
to be able to take either point of view according to a given situation.

\subsection{Stability under finite Limits}

\begin{definition}
  Let $M \subseteq \sepa$. The category $\mass$ of
  {\tmem{$M$-assemblies}} is the full subcategory of
  $\asma$ with objects assemblies
  valued in $M$.
\end{definition}

\begin{remark}
  $\asm{M} \cong \mathbb{P} [M]$
\end{remark}

Recall that in an implicative algebra, the product $\sqcap$ (with respect to
the entailment preorder) is given by the second-order encoding of pairs.

\begin{definition}
  A subset $M \subseteq \sepa$ is {\tmem{algebraic}} if if it is closed under
  $\sqcap$.
\end{definition}

\begin{remark}
  $\mass$ is closed under finite limits if $M$ is algebraic.
\end{remark}

\begin{example}
  \label{ex:alg}
  $\mathcal{A}= \left( \pows{P}, \subseteq, \rightarrow, \sepa = \ppows{P} \right)$ with $P$
  a combinatory algebra. The subset
  \begin{eqnarray*}
    M & \assign & \{ \{ p \} |p \in P \}
  \end{eqnarray*}
  of the separator is manifestly algebraic as in this case we have $\{ p_1 \}
  \sqcap \{ p_2 \} = \{ \langle p_1, p_2 \rangle \}$.
\end{example}

\subsection{Implicative Existence}

Before we proceed further we need to compile some further useful facts about
the operation of existence in implicative algebras.

\begin{remark}
  \label{rem:eta}Recall that the implicative existence is given by
  second-order encoding as
  \begin{eqnarray*}
    \siex U & = & \bigM_{c \in \mathcal{A}} \bigM_{u \in U}
    ((u \rightarrow c) \rightarrow c)
  \end{eqnarray*}
  \begin{enumeratenumeric}
    \item We have $\vdash \lam{z}{z \nocomma u} : \forall c : \mathcal{A}.
    \forall u : U. (u \rightarrow c) \rightarrow c$, hence $\intpr{\lam{z}{z
    \nocomma u}} \preccurlyeq \siex U$ \ for all $u \in U$, so in particular
    $\siex U \in \sepa$ if $U \cap \sepa \neq \varnothing$.

    \item A consequence of the above item is that if we have $\chi \assign a
    \rightarrow \siex U$ then for any $a' \preccurlyeq a$ there is an $u \in
    U$ such that $\chi a' \preccurlyeq \ilam{z}{z u}$. We systematically
    exploit this fact, in particular when building trackers.
  \end{enumeratenumeric}
\end{remark}

\subsection{Valuations}

\begin{definition}
  Let $X$ be a set and $\varnothing \neq M \subseteq \sepa$. An
  {\tmem{$M$-valuation of $X$}} (or just {\tmem{valuation}} if $M$ and $X$ are
  understood) is map $\nu : X \rightarrow \ppows{M}$.
\end{definition}

\begin{remark}
  \label{rem:fib-val}Let $X$ be a set and $M \subseteq \sepa$. We have the
  predicate
  \begin{eqnarray*}
    \tmop{ex} : \ppows{M} & \longrightarrow & \sepa\\
    U & \mapsto &  \siex U
  \end{eqnarray*}
  in $\mathbb{P} [\sepa]_{\ppows{M}}$. Given a valuation $\nu : X
  \rightarrow \ppows{M}$, we have the predicate $\nu^{\ast} (\tmop{ex}) \in
  \mathbb{P} \left[ \sepa \right]_X$ where
  \begin{eqnarray*}
    \nu^{\ast} (\tmop{ex})_x & = & (\tmop{ex} \circ \nu) (x)\\
    & = & \tmop{ex} (\nu (x))\\
    & = & \siex (\nu (x))
  \end{eqnarray*}
  so
  \begin{eqnarray*}
    \nu^{\ast} (\tmop{ex}) & = & {\left( \siex (\nu (x)) \right)_{x \in X}}
  \end{eqnarray*}
\end{remark}

\begin{notation}
  Let $\nu : X \rightarrow \ppows{M}$ be a valuation. We shall use the
  notation $\siex \nu : = \nu^{\ast} (\tmop{ex})$.
\end{notation}

\begin{lemma}
  \label{lem:induced-valuation}Any morphism $f : X \rightarrow Y$ in $\asm{M}$
  gives rise to the valuation
  \begin{eqnarray*}
    \nu_f : \car{\tmop{Im} (f)} & \longrightarrow & \ppows{M}\\
    {}[x] & \mapsto & \left\{ \exi{X} (x') |x' \in [x] \right\}
  \end{eqnarray*}
  and we have $\exi{\tmop{Im} (f)} = \siex \nu_f$.
\end{lemma}

\begin{notation}
  $[x] \assign f^{- 1} (f (x))$.
\end{notation}

\begin{proof}
  We have $\car{\tmop{Im} (f)} = \{ [x] |x \in X \} $ where $\exi{\tmop{Im}
  (f)} ([x]) = \iex{x' \in [x]} \exi{X} (x')$, hence
  \begin{eqnarray*}
    (\siex \nu_f)_{[x]} & = & (\nu_f^{\ast} \tmop{ex})_{[x]}\\
    & = & \iex{x' \in [x]} \exi{X} (x')
  \end{eqnarray*}

\end{proof}

\begin{proposition}
  \label{prop:induced-valuation}Let $X$ be a set and $\nu : X \rightarrow
  \ppows{M}$ be a valuation. Let $\hat{X} \in \mass$ be the assembly given by
  \begin{itemizeminus}
    \item $\car{\hat{X}} \assign \sum_{x \in X} \nu (x)$;

    \item $\exi{\hat{X}} (x, m) \assign m$ for all $x \in X$ and $m \in \nu(x)$.
  \end{itemizeminus}
  and $\check{X} \in \mass$ be the assembly given by
  \begin{itemizeminus}
    \item $\car{\check{X}} \assign X$;

    \item $\exi{\check{X}} (x) \assign m_0$ for all $x \in X$ and some $m_0
    \in M$.
  \end{itemizeminus}
  The obvious surjective map $g_{\nu} : \car{\hat{X}}
  \twoheadrightarrow \car{\check{X}}$ is tracked by
  $(\lambda x.m_0)^{\mathcal{A}}$
  and we have the isomorphism
  \begin{eqnarray*}
    {}[-]_{\nu} : \left( X, \siex \nu \right) & \longrightarrowlim^{\cong} &
    \left( \tmop{Im} (g_{\nu}), \siex g_{\nu} \right)\\
    x & \mapsto & g_{\nu}^{- 1} (x)
  \end{eqnarray*}
  in $\asma$.
\end{proposition}

\begin{proof}
  Consider the image factorisation

  \begin{center}
    \raisebox{-0.5\height}{\includegraphics[width=14.8909550045914cm,height=2.38013249376886cm]{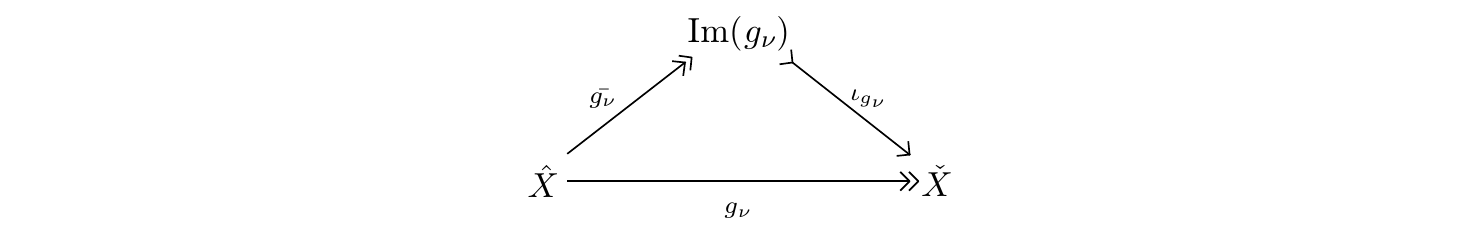}}
  \end{center}

  We have $\tmop{Im} (g_{\nu}) = \{ g_{\nu}^{- 1} (x) |x \in X \}$ as
  \begin{eqnarray*}
    {}[x_m] & = & g_{\nu}^{- 1} (g_{\nu} (x_m))\\
    & = & g_{\nu}^{- 1} (x)
  \end{eqnarray*}
  for all $m \in \nu (x)$, hence the bijection
  \begin{eqnarray*}
    {}[-]_{\nu} : X & \longrightarrowlim^{\cong} & \tmop{Im} (g_{\nu})\\
    x & \mapsto & g_{\nu}^{- 1} (x)
  \end{eqnarray*}
  We have
  \begin{eqnarray*}
    \nu_{g_{\nu}} (g_{\nu}^{- 1} (x)) & = & \left\{ \exi{\hat{X}} (x_m) |x_m
    \in g_{\nu}^{- 1} (x) \right\}\\
    & = & \{ m \in M|x_m \in g_{\nu}^{- 1} (x) \}\\
    & = & \nu (x)
  \end{eqnarray*}
  hence
  \begin{eqnarray*}
    \exi{\tmop{Im} (g_{\nu})} ([x]_{\nu}) & = & {\left( \siex \nu_{g_{\nu}}
    \right)_{g_{\nu}^{- 1} (x)}}   \text{\qquad Lemma
    \ref{lem:induced-valuation}}\\
    & = & \siex (\nu_{g_{\nu}} (g_{\nu}^{- 1} (x)))\\
    & = & \siex \nu (x)
  \end{eqnarray*}
  and therefore
  \begin{eqnarray*}
    \bigM_{x \in X} \left( \left( \siex \nu \right)_x \rightarrow
    \exi{\tmop{Im} (g_{\nu})} ([x]_{\nu}) \right) & = & \bigM_{x \in
    X} \left( \left( \siex \nu \right)_x \rightarrow \left( \siex \nu
    \right)_x \right)\\
    & = & \bigM_{x \in X} \left( \exi{\tmop{Im} (g_{\nu})}
    ([x]_{\nu}) \rightarrow \left( \siex \nu \right)_x \right)
  \end{eqnarray*}
  But $\bigM_{x \in X} \left( \left( \siex \nu \right)_x \rightarrow
  \left( \siex \nu \right)_x \right) \in \sepa$, hence $[-]_{\nu}$ and \
  $[-]_{\nu}^{- 1}$ are tracked.

\end{proof}

\section{Regular Completions of Implicative Assemblies}
\label{sec:regular-completion-asm}

\subsection{Density}

\begin{definition}
  A non-empty subset $M \subseteq \sepa$ of the separator is
  {\tmem{dense}} in $\sepa$ (or just {\tmem{dense}} if $\sepa $ is understood)
  if there is a valuation $\nu : \sepa \rightarrow \ppows{M}$ such that
  $\tmop{id}_{\sepa} \cong_{\sepa} \siex \nu$.
\end{definition}

\begin{theorem}
  \label{th:dense}Let $M \subseteq \sepa$ be an algebraic subset. The
  following are equivalent
  \begin{enumerateroman}
    \item $M$ is dense;
    \item for any $X \in \ens$ and any $u \in \mathbb{P} \left[ \sepa
    \right]_X$ there is a valuation $\nu : X \rightarrow \ppows{M}$ such that
    $u \cong_X \siex \nu$;
    \item the functor $U : \reglex{\left( \mass \right)} \rightarrow
    \asma$ is essentially surjective.
  \end{enumerateroman}
\end{theorem}

\begin{proof}
  We need to insist that $M$ is algebraic only in order to be able to talk
  about $\reglex{\left( \mass \right)}$, otherwise it is {\tmem{deus ex
  machina}} behind the stage.

  $(i) \Longrightarrow (i \nocomma i)$ There is a valuation $\nu : \sepa
  \rightarrow \ppows{M}$ of $\sepa$ such that $\tmop{id}_{\sepa} \cong_{\sepa}
  \siex \nu$. Let $u \in \sfibr{X}$. We have
  \begin{eqnarray*}
    \tmop{id}_{\sepa} & = & i_s^{\ast} (\tmop{id}_{\mathcal{A}})
  \end{eqnarray*}
  with $i_{\sepa} : \sepa \rightarrowtail \mathcal{A}$ the insertion map,
  hence
  \begin{eqnarray*}
    u & = & u^{\ast} (\tmop{id}_{\mathcal{A}}) \hspace{4em}
    \tmop{id}_\mathcal{A} \in \fibr{\mathcal{A}}
    \text{is a split generic object}\\
    & = & \left( i_{\sepa} \circ u \right)^{\ast} (\tmop{id}_{\mathcal{A}})\\
    & = & u^{\ast} (i_s^{\ast} (\tmop{id}_{\mathcal{A}})) \hspace{3em}
    \mathbb{P} [\mathcal{A}] \rightarrow \ens \text{ is a split fibration}\\
    & = & u^{\ast} (\tmop{id}_{\sepa})\\
    & \cong_X & u^{\ast} \left( \siex \nu \right) \qquad
    \text{functoriality}\\
    & = & \siex (\nu \circ u)
  \end{eqnarray*}
  so the valuation $\nu \circ u : X \rightarrow \ppows{M}$ witnesses the fact
  that $M$ is dense.

  $(i \nocomma i) \Longrightarrow (\nocomma i \nocomma i \nocomma i)$ Let $X
  \in \asm{M}$. Assume a valuation $\nu : X \rightarrow \ppows{M}$ such that
  $\exi{X} \cong_X \siex \nu$. There is a

  morphism $g_{\nu}$ such that $\Xi (\tmop{Im} (g_{\nu})) = \siex \nu$
  (c.f. Lemma \ref{lem:induced-valuation}). We furthermore have $\Xi (X) =
  \exi{X}$, hence the isomorphism
  \begin{eqnarray*}
    \Xi^{- 1} \left( \exi{X} \cong_X \siex \nu \right) : X &
    \longrightarrowlim^{\cong} & \tmop{Im} (g_{\nu})
  \end{eqnarray*}
  $(i \nocomma i \nocomma i) \Longrightarrow (i \nocomma)$ We have $\Xi^{- 1}
  \left( \tmop{id}_{\sepa} \right) = \left( \sepa, \tmop{id}_{\sepa} \right)$.
  There is by hypothesis a morphism $f : X \rightarrow Y$ in $\asm{M}$ such
  that $\sepa \cong \tmop{Im} (f)$. Let $k : \sepa
  \longrightarrowlim^{\cong} \tmop{Im} (f)$ be a witnessing iso. We have
  \begin{eqnarray*}
    \Xi (k) & = & \left( k, \tmop{id} \text{ } \nocomma \sqsubseteq_{\sepa}
    \text{ } k^{\ast} \siex \nu_f \right)\\
    & = & \left( k, \tmop{id} \text{ } \nocomma \sqsubseteq_{\sepa} \text{ }
    \siex (\nu_f \circ k) \right)
  \end{eqnarray*}
  and
  \begin{eqnarray*}
    \Xi (k^{- 1}) & = & \left( k^{- 1}, \siex \nu_f \text{ }
    \sqsubseteq_{\tmop{Im} (f)} \text{ } (k^{- 1})^{\ast} \tmop{id}_{\sepa}
    \right)\\
    & = & \left( k^{- 1}, \siex \nu_f \text{ } \sqsubseteq_{\tmop{Im} (f)}
    \text{ } k^{- 1} \right)
  \end{eqnarray*}
  Notice that $k^{- 1} : \tmop{Im} (f) \rightarrow \sepa$ has the ``right
  type'' here in the sense that $k^{- 1} \in \mathbb{P} \left[ \sepa
  \right]_{\tmop{Im} (f)}$. Reindexing the vertical morphism $\siex \nu_f
  \text{ } \sqsubseteq_{\tmop{Im} (f)} \text{ } k^{- 1}$ by $k$ yields
  \begin{eqnarray*}
    k^{\ast} \left( \siex \nu_f \text{ } \sqsubseteq_{\tmop{Im} (f)} \text{ }
    k^{- 1} \right) & = & k^{\ast} \left( \siex \nu_f \right) \text{ }
    \sqsubseteq_{\sepa} \text{ } k^{\ast} (k^{- 1})\\
    & = & \siex (\nu_f \circ k) \text{ } \sqsubseteq_{\sepa} \text{ }
    \tmop{id}_{\sepa}
  \end{eqnarray*}
  We thus have $\tmop{id}_{\sepa} \text{ } \cong_{\sepa} \text{ } \siex
  (\nu_f \circ k)$ by virtue of the valuation $\nu_f \circ k$.
\end{proof}

\begin{example}
  \label{ex:single-dense} In example \ref{ex:alg} we considered an implicative
  algebra $\mathcal{A}= \left( \pows{P}, \subseteq, \rightarrow, \sepa
  \right)$ constructed from an applicative structure $P$ and showed that the
  set $M \assign \{ \{ p \} |p \in P \}$ of all singletons is algebraic. Now
  $\sepa = \ppows{P}$ here so we have the {\tmem{trivial valuation}}
  \begin{eqnarray*}
    \nu : \sepa & \longrightarrow & \ppows{M}\\
    U & \mapsto & \{ \{ m \} |m \in U \}
  \end{eqnarray*}
  which yields $\siex \nu_U = U$, hence
  \begin{eqnarray*}
    \bigM_{U \in \sepa} \left( U \rightarrow \left( \siex \nu
    \right)_U \right) & = & \bigcap_{U \in \sepa} (U \rightarrow U)\\
    & = & \bigM_{U \in \sepa} \left( \left( \siex \nu \right)_U
    \rightarrow U \right)
  \end{eqnarray*}

  But $\bigcap_{U \in \sepa} (U \rightarrow U) \neq \varnothing$ as it
  contains $\tmop{id}$, hence $\tmop{id} \cong_{\sepa} \siex \nu$.
\end{example}

\begin{definition}
  An implicative structure $\mathcal{A}$ is {\tmem{compatible with joins}}
  provided
  \begin{eqnarray*}
    \bigM_{a \in A} (a \rightarrow b) & = & \left( \bigcurlyvee_{a
    \in A} a \right) \rightarrow b
  \end{eqnarray*}
  for all $A \subseteq \mathcal{A}$ and $b \in \mathcal{A}$. An implicative
  algebra is compatible with joins if it is the case for the underlying
  implicative structure.
\end{definition}

\begin{remark}
  In an implicative algebra $\mathcal{A}$ compatible with joins the operations
  $\siex$ and $\bigcurlyvee$are the same.
\end{remark}

\begin{proposition}
  et $\mathcal{A}$ be an implicative algebra compatible with joins. An
  algebraic subset $M \subseteq \sepa$ is dense provided any element of
  $\sepa$ is a join of elements of $M$.
\end{proposition}

\begin{proof}
  For each $a \in \mathcal{A}$ we can choose a subset $M_a \subseteq M$ such
  that $a = \bigcurlyvee M_a$ by virtue of $(i)$, hence there is the trivial
  valuation
  \begin{eqnarray*}
    \nu : \mathcal{A} & \longrightarrow & \ppows{M}\\
    a & \mapsto & M_a
  \end{eqnarray*}
  We have
  \begin{eqnarray*}
    \left( \siex \nu \right)_a & = & \siex M_a\\
    & = & \bigcurlyvee M_a \qquad \text{compatible with joins}\\
    & = & a
  \end{eqnarray*}
  hence
  \begin{eqnarray*}
    \bigM_{a \in A} \left( a \rightarrow \left( \siex \nu \right)_a
    \right) & = & \bigM_{a \in \mathcal{A}} (a \rightarrow a)\\
    & = & \bigM_{a \in A} \left( \left( \siex \nu \right)_a
    \rightarrow a \right)
  \end{eqnarray*}
  But $\tmop{id} \preccurlyeq \bigM_{a \in \mathcal{A}} (a
  \rightarrow a)$ hence $\bigM_{a \in \mathcal{A}} (a \rightarrow a)
  \in \sepa$, which implies that $M$ is dense.
\end{proof}

\subsection{Compactness}

\begin{lemma}
  \label{lem:pre-compact}Let $M \subseteq \sepa$ and let $g : A
  \rightarrow B$ be a morphism in $\mass$. The following are equivalent
  \begin{enumerateroman}
    \item any morphism $k : X \rightarrow \tmop{Im} (g)$ in
    $\asma$ with $X{\in}{\mass}$ admits a
    lift

    \begin{center}
      \raisebox{-0.5\height}{\includegraphics[width=14.8909550045914cm,height=2.87590187590188cm]{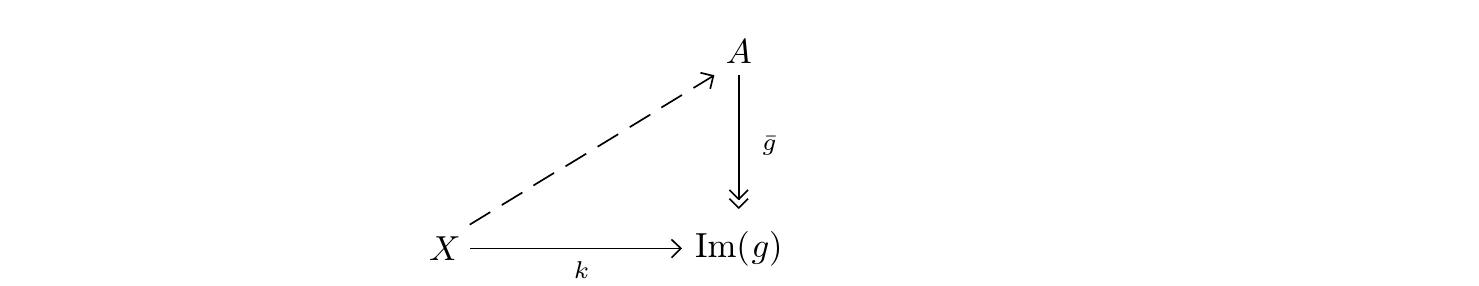}}
    \end{center}

    \item any morphism $l \circ \bar{f} : X \rightarrow \tmop{Im} (g)$ in
    $\asma$ where
    \begin{itemizeminus}
      \item $f : X \rightarrow Y$ is a morphism in $\mass$;

      \item $l : \tmop{Im} (f) \rightarrow \tmop{Im} (g)$ is a morphism in
      $\asma$
    \end{itemizeminus}
    admits a lift

    \begin{center}
      \raisebox{-0.5\height}{\includegraphics[width=14.8909550045914cm,height=2.97507542962088cm]{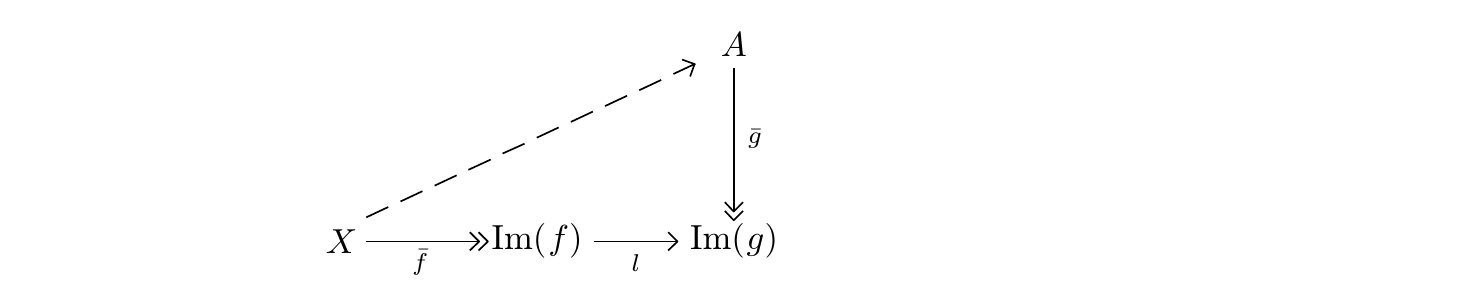}}
    \end{center}
  \end{enumerateroman}
\end{lemma}

\begin{proof}
  The implication $(i) \Rightarrow (i \nocomma i)$ is trivial. Assume $(i
  \nocomma i)$ and let $k : X \rightarrow \tmop{Im} (g)$ be a morphism in
  $\asma$. We have the diagram

  \begin{center}
    \raisebox{-0.5\height}{\includegraphics[width=14.8909550045914cm,height=4.85904499540863cm]{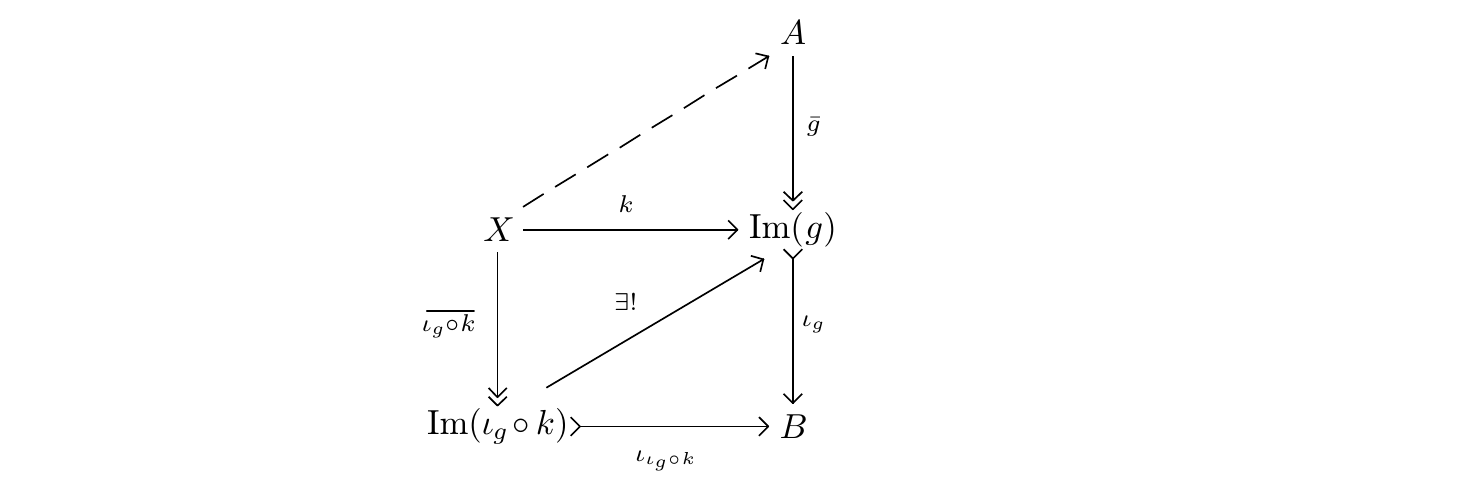}}
  \end{center}

  in $\asma$ with
  \begin{itemizeminus}
    \item $X, B \in \mass$;

    \item $\exists !$ the unique lift of the lower square (given that regular
    epis and monos form an orthogonal factorisation system)
  \end{itemizeminus}
  hence $k = (\exists !) \circ \overline{\iota_g \circ k}$ lifts.
\end{proof}

\begin{definition}
  $M \subseteq \sepa$ is {\tmem{compact}} if for any $u \in \mfibr{X}$
  and any valuation $\nu : X \rightarrow \ppows{M}$ such that
  \begin{eqnarray*}
    u & \sqsubseteq_X & \siex \nu
  \end{eqnarray*}
  there is $b \in \mfibr{X}{M}$ with $b_x \in \nu (x)$ for all $x \in X$ such
  that
  \[ u \text{ } \sqsubseteq_X \text{ } b \text{ } \sqsubseteq_X \text{ } \siex
     \nu \]
\end{definition}

\begin{lemma}
  \label{lem:compact}Let $M \subseteq \sepa$. Assuming choice the
  following are equivalent
  \begin{enumerateroman}
    \item $M$ is compact;

    \item for any morphisms $g : A \rightarrow B$ in $\mass$ and $k : X
    \rightarrow \tmop{Im} (g)$ in
    $\asma$ with $X \in \mass$ the
    diagram

    \begin{center}
      \raisebox{-0.5\height}{\includegraphics[width=14.8909550045914cm,height=2.87590187590188cm]{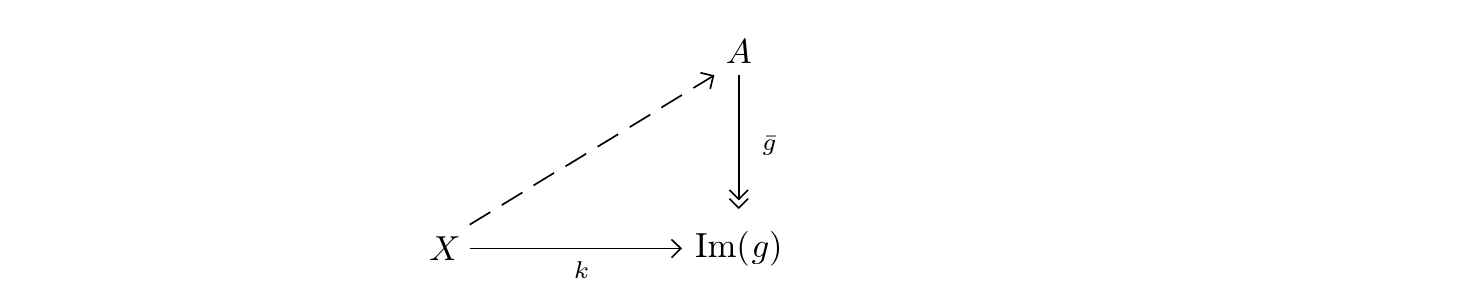}}
    \end{center}

    admits a lift.
  \end{enumerateroman}
\end{lemma}

\begin{proof}
  $(i \nocomma) \Rightarrow (i \nocomma i)$ Let $g : A \rightarrow B$ be a
  morphism in $\mass$ and $k : X \rightarrow \tmop{Im} (g)$ be a morphism in
  $\asma$ with $X \in \mass$. The
  induced valuation $\nu_g : \tmop{Im} (g) \rightarrow \ppows{M}$ verifies
  \begin{eqnarray*}
    \siex \nu_g & = & \exi{\tmop{Im} (g)} \qquad (\star)
  \end{eqnarray*}
  (c.f. Lemma \ref{lem:induced-valuation}). The map $\nu \assign \nu_{\bar{g}}
  \circ k$ is a valuation as well. Furthermore

  \begin{eqnarray*}
    \bigM_{x \in X} \left( \exi{X} (x) \rightarrow \siex \nu (x)
    \right) & = & \bigM_{x \in X} \left( \exi{X} (x) \rightarrow
    \siex \nu_g (k (x)) \right)\\
    & = & \bigM_{x \in X} \left( \exi{X} (x) \rightarrow
    \exi{\tmop{Im} (g)} (k (x)) \right)\\
    & \in & \sepa
  \end{eqnarray*}
  since $k$ is tracked, hence $\exi{X} \sqsubseteq_X \siex \nu$. $M$ is
  discrete by hypothesis, hence there is $b \in \mfibr{X}{M}$ with $b_x \in
  \nu (x)$ for all $x \in X$ such that
  \[ \exi{X} \text{ } \sqsubseteq_X \text{ } b \text{ } \sqsubseteq_X \text{ }
     \siex \nu \]
  We have
  \begin{eqnarray*}
    \Xi \left( \exi{X} \text{ } \sqsubseteq_X \text{ } b \text{ } \right) & =
    & \left( X, \exi{X} \right) \longrightarrowlim^{\tmop{id}} (X, b)
  \end{eqnarray*}
  On the other hand
  \begin{eqnarray*}
    \left( \siex \nu \right) (x) & = & \siex (\nu_{\bar{g}} \circ k) (x)\\
    & = & \siex \nu_{\bar{g}} (k (x))
  \end{eqnarray*}
  for all $x \in X$, hence $\siex \nu = k^{\ast} \siex \nu_{\bar{g}}$ which
  entails
  \begin{eqnarray*}
    \Xi \left( b \text{ } \sqsubseteq_X \text{ } \siex \nu \right) & = & \Xi
    \left( b \text{ } \sqsubseteq_X \text{ } k^{\ast} \siex \nu_{\bar{g}}
    \right)\\
    & = & (X, b) \longrightarrowlim^k \left( \tmop{Im} (f), \siex \nu_g
    \right)
  \end{eqnarray*}
  This in turn entails that $k$ factors as
  \[ \left( X, \exi{X} \right) \longrightarrowlim^{\tmop{id}} (X, b)
     \longrightarrowlim^k \left( \tmop{Im} (f), \siex \nu_g \right) \]
  in $\asma$. Now $b_x \in \nu (x) =
  \left\{ \exi{A} (a) |a \in k (x) \right\}$ for each $x \in X$, hence there
  is $\tilde{a} \in k (x)$ such that $b_x = \exi{A} (\tilde{a})$. Notice that
  while the {\tmem{value}} $\exi{A} (\tilde{a})$ of $b_x$ is determined,
  $\tilde{a}$ itself does not need to be unique. Let
  \begin{eqnarray*}
    \tilde{A}_x & : = & \left\{ \tilde{a} \in k (x) |b_x = \exi{A} (\tilde{a})
    \right\}
  \end{eqnarray*}
  be the set of those $\tilde{a}$. Assuming choice there is a map $l : X
  \rightarrow A$ such that $l (x) \in \tilde{A}_x$ for all $x \in X$. \
  Furthermore
  \begin{eqnarray*}
    \bigM_{x \in X} \left( b_x \rightarrow \exi{A} (l (x)) \right) &
    = & \bigM_{x \in X} \left( \exi{A} (\tilde{a}) \rightarrow
    \exi{A} (\tilde{a}) \right)\\
    & \in & \sepa
  \end{eqnarray*}
  so $l$ is tracked. Finally we have $\bar{g} \circ l = k$ as maps since $k
  (x) = \bar{g}  (a)$ for {\tmem{any}} $a \in k (x)$ hence the lift

  \begin{center}
    \raisebox{-0.5\height}{\includegraphics[width=14.8909550045914cm,height=3.96664698937426cm]{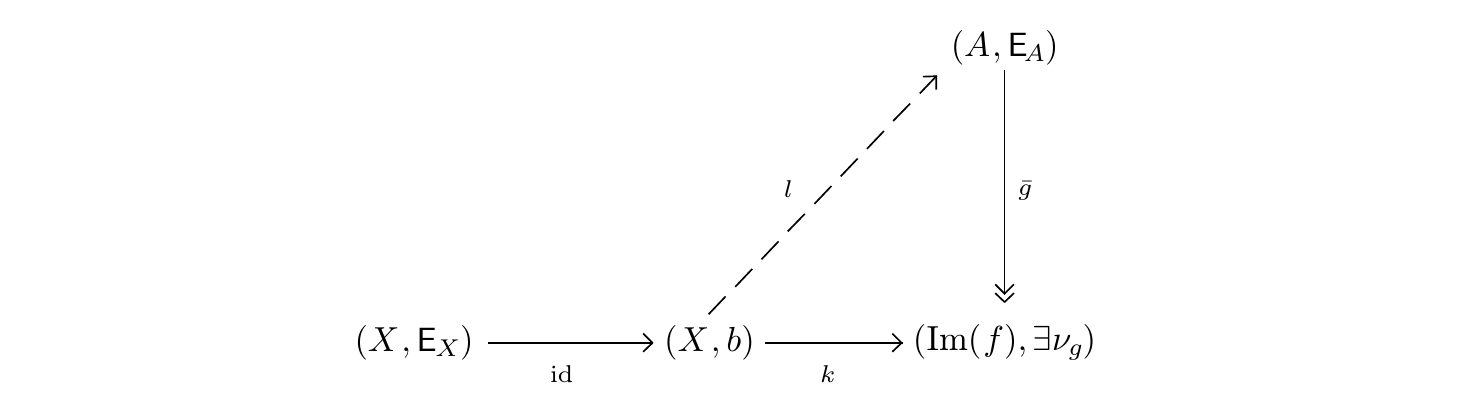}}
  \end{center}

  $(i \nocomma i) \Rightarrow (\nocomma i)$ Let $u \in \mathbb{P} [M]_X$ and
  $\nu : X \rightarrow \ppows{M}$ be a valuation such that $\bigM_{x
  \in X} \left( u_x \rightarrow \left( \siex \nu \right)_x \right) \in
  \sepa$. We thus have $u \sqsubseteq_X \siex \nu$ and further $\Xi^{-
  1} \left( u \sqsubseteq_X \siex \nu \right) = \tmop{id} : (X, u) \rightarrow
  \left( X, \siex \nu \right)$. Composing the latter with $[-]_{\nu}$ (c.f
  Proposition \ref{prop:induced-valuation}) yields the morphism $[-]_{\nu} :
  (X, u) \rightarrow \left( \tmop{Im} (g_{\nu}), \exi{\tmop{Im} (g_{\nu})}
  \right)$ in $\asma$. The latter has a
  lift

  \begin{center}
    \raisebox{-0.5\height}{\includegraphics[width=14.8909550045914cm,height=2.97507542962088cm]{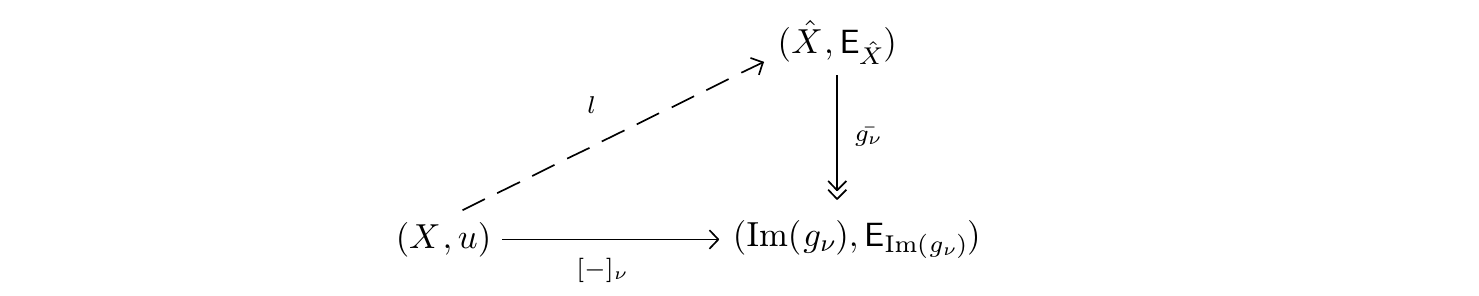}}
  \end{center}

  Let $b \in \mathbb{P} [M]_X$ be given by $b_x \assign \exi{\hat{X}} (l
  (x))$. We have by construction $l (x) = x_m$ for some $m \in \nu (x)$, hence
  \begin{eqnarray*}
    \exi{\hat{X}} (l (x)) & = & \exi{\hat{X}} (x_m)\\
    & = & m
  \end{eqnarray*}
  so $b_x \in \nu (x)$ for all $x \in X$. Now $l$ is tracked so
  \begin{eqnarray*}
    \bigM_{x \in X} (u_x \rightarrow b_x) & = & \bigM_{x \in
    X} \left( \exi{X} \rightarrow \exi{\hat{X}} (l (x)) \right)\\
    & \in & \sepa
  \end{eqnarray*}

\end{proof}

\begin{theorem}
  \label{th:compact}Let $M \subseteq \sepa$ be an algebraic subset.
  Assuming choice the following are equivalent
  \begin{enumerateroman}
    \item  $M$ is {\tmem{compact}};
    \item $U : {\mass}_{\tmop{reg}} \rightarrow \asma$ is full.
  \end{enumerateroman}
\end{theorem}

\begin{proof}
  $M$ being compact is equivalent to the assertion that any morphism $l \circ
  \bar{f} : X \rightarrow \tmop{Im} (g)$ where
  \begin{enumerateroman}
    \item $f : X \rightarrow Y$ and $g : A \rightarrow B$ are morphisms in
    $\mass$;

    \item $l : \tmop{Im} (f) \rightarrow \tmop{Im} (g)$ is a morphism in
    $\asma$
  \end{enumerateroman}
  admits a lift

  \begin{center}
    \raisebox{-0.5\height}{\includegraphics[width=14.8909550045914cm,height=2.97507542962088cm]{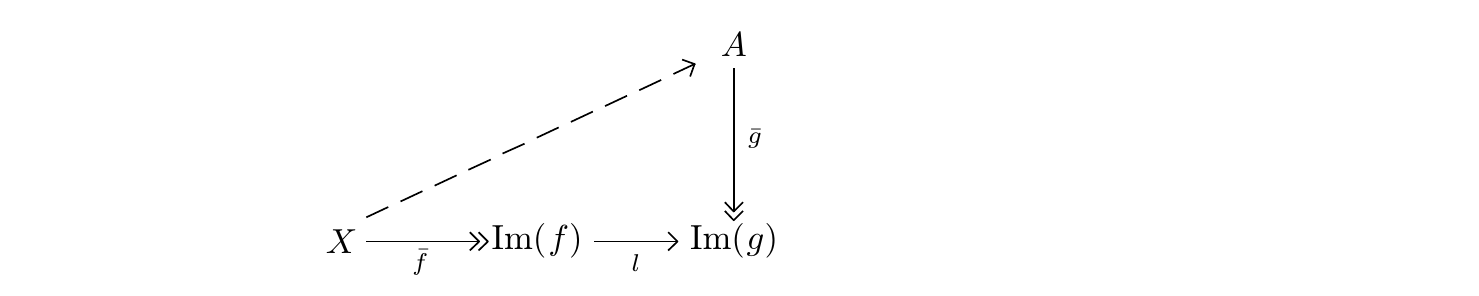}}
  \end{center}

  (c.f. Lemma \ref{lem:pre-compact} and \ref{lem:compact}), which is obviously
  equivalent to the assertion that $U : {\mass}_{\tmop{reg}} \rightarrow
  \asma$ is full (c.f. Section \ref{sec:regular}).
\end{proof}

\begin{example}
  We observed that the set of all singletons $M \assign \{ \{ p \} |p \in P
  \}$ is algebraic (c.f. Example \ref{ex:alg}) and dense (c.f. Example
  \ref{ex:single-dense}). Taking it from there we now show that $M$ is compact
  as well. Let $u \in \mathbb{P} [M]_X$ be a family of singletons $u_x = \{
  \check{u}_x \}$ and $\nu : X \rightarrow \ppows{M}$ be a valuation. Suppose
  $\bigM_{x \in X} \left( u_x \rightarrow \left( \siex \nu \right)_x
  \right) \in \sepa$. Unravelling yields
  \begin{eqnarray*}
    \bigM_{x \in X} (u_x \rightarrow (\siex \nu)_x) & = &
    \bigcap_{x \in X} \left( \{ \check{u}_x \} \rightarrow \bigcup_{s \in \nu
    (x)} s \right)\\
    & = & \bigcap_{x \in X} \left\{ c \in P|c \cdot \check{u}_x \in \bigcup
    \nu (x) \right\}\\
    & = & \bigcap_{x \in X} \{ c \in P|c \cdot \check{u}_x \in \nu (x) \}
    \qquad \text{elements of $\nu (x)$ are singletons}\\
    & \neq & \varnothing
  \end{eqnarray*}
  Let $c_0 \in \bigcap_{x \in X} \{ c \in P|c \cdot \check{u}_x \in \nu (x)
  \}$ and further $b \in \mathbb{P} [M]_X$ given by $b_x \assign \{ c \cdot
  \check{u}_x \}$. We have $b_x \in \nu (x)$ for all $x \in X$ and further
  \begin{eqnarray*}
    \bigM_{x \in X} (u_x \rightarrow b_x) & = & \bigcap_{x \in X}
    (u_x \rightarrow b_x)\\
    & = & \bigcap_{x \in X} (\{ \check{u}_x \} \rightarrow \{ c_0 \cdot
    \check{u}_x \})\\
    & \neq & \varnothing
  \end{eqnarray*}
\end{example}

\subsection{Generators}

\begin{definition}
  An algebraic subset $M \subseteq \sepa$ is a {\tmem{generator}} of
  $\sepa$ if it is dense and compact.
\end{definition}

\begin{theorem}
  \label{th:hilbert}Let $M \subseteq \sepa$ be an algebraic subset. The
  following are equivalent
  \begin{enumerateroman}
    \item $M$ is a generator of $\sepa$;

    \item $\asma$ is the regular
    completion of $\mass$.
  \end{enumerateroman}
\end{theorem}

\begin{proof}
  Theorem \ref{th:dense} and \ref{th:compact}.
\end{proof}

\begin{corollary}
  If one and thus both of the equivalent assertions of Theorem
  \ref{th:hilbert} hold, the closure of $\mass$ under isomorphisms is
  $\asma$'s subcategory of projective
  objects.
\end{corollary}

\begin{corollary}
  If $M$ and $M'$ are generators then $\asm{M} \simeq \asm{M'}$.
\end{corollary}

\section{Exact Completions}\label{sec:exact}

\begin{definition} Let $\mathbb{D}$ be a regular category
    \begin{enumeratenumeric}
        \item An equivalence relation in $\mathbb{D}$ is
        {\tmem{effective}} if it is a kernel pair.

        \item $\mathbb{D}$ is {\tmem{exact}} if every
        equivalence relation is effective.
  \end{enumeratenumeric}
\end{definition}

\begin{example}
  Any topos among others.
\end{example}

An exact category in the above sense is sometimes called {\tmem{Barr-exact}}
as there is quite a different notion of exactness due to Quillen.
We won't need the distinction here as the only
exact categories in sight are those in the sense of Barr.

\begin{notation}
  We shall write
  \begin{itemizeminus}
    \item $\Ex$ for the 2-category of exact categories, exact funtors and
    their natural transformations;

    \item $| - | : \Ex \rightarrow \Reg$ for the forgetful $2$-functor.
  \end{itemizeminus}
\end{notation}

It is part of the lore that the forgetful 2-functor $| - |$ has a left
biadjoint $(-)_{\tmop{ex} / \tmop{reg}} : \Reg \rightarrow \Ex$. This entails
that the forgetful 2-fiunctor $| - | : \Ex \rightarrow \Lex$ has a left
biadjoint $(-)_{\tmop{ex} / \tmop{lex}} : \Lex \rightarrow \Ex$ as well since
biadjoints compose.

Let $\mathbb{C}$ be a lex category.

\begin{definition}
  We call $\exlex{\mathbb{C}}$ $\mathbb{C}$'s {\tmem{exact completion}}.
\end{definition}

There is an embedding $\y : \mathbb{C} \rightarrow \exlex{\mathbb{C}}$
enjoying a universal property structurally identical to the one of regular
completions (c.f. Section \ref{sec:regular}). There is a well-known direct
construction of $\mathbb{C}_{\tmop{ex} / \tmop{lex}}$ {\cite{carboni1982free}}
which we now recall.

Let $\mathbb{C}$ be a category.

\begin{definition}
  Let $\mathcal{Q}$ be a category $s, t : Q_1 \rightrightarrows Q_0$ with two
  objects and two parallel morphisms. A quiver in $\mathbb{C}$ is a functor
  $\mathcal{Q} \rightarrow \mathbb{C}$.
\end{definition}

As the only quivers of interest here are those internal to a (typically lex)
category $\mathbb{C}$, we shall dispense with the qualifier ``in
$\mathbb{C}$''.

\begin{notation}
  Let $X$ be a quiver.
  \begin{enumeratenumeric}
    \item $X (x, x') \assign \pr{s_X}{t_X}^{- 1} (x, x')$.

    \item We shall write $e : x \rightsquigarrow x'$ for an edge $e \in X (x,
    x')$.
  \end{enumeratenumeric}
\end{notation}

\begin{definition}
  Let $X, Y$ be quivers.
  \begin{enumeratenumeric}
    \item A quiver morphism $f : X \rightarrow Y$ is given by morphisms $f_0 :
    X_0 \rightarrow Y_0$ and $f_1 : X_1 \rightarrow Y_1$ such that

    \begin{center}
      \raisebox{-0.5\height}{\includegraphics[width=14.8909550045914cm,height=2.97507542962088cm]{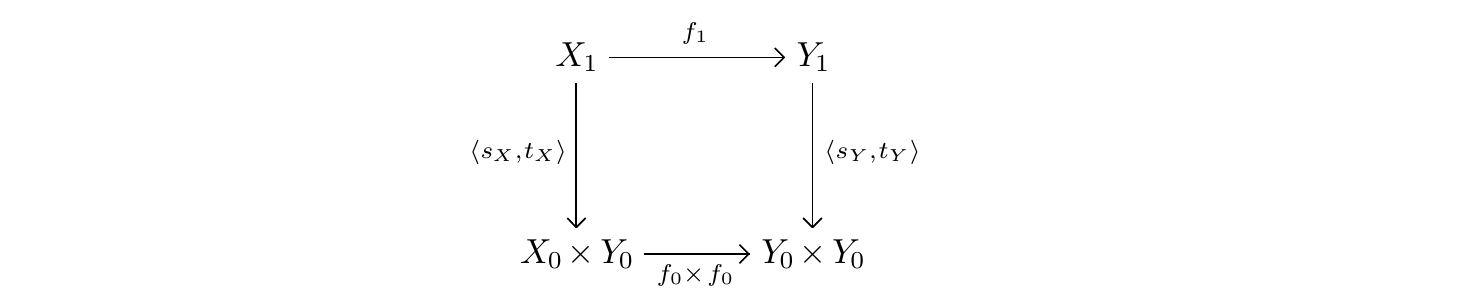}}
    \end{center}

    commutes.

    \item Let $f, g : X \rightarrow Y$ be quiver morphisms. A
    {\tmem{homotopy}} from $f$ to $g$ is a morphism $h : X_0 \rightarrow Y_1$
    such that the morphism $\langle f_0, g_0 \rangle : X_0 \rightarrow Y_0
    \times Y_0$ has a lift

    \begin{center}
      \raisebox{-0.5\height}{\includegraphics[width=14.8909550045914cm,height=3.37170405352224cm]{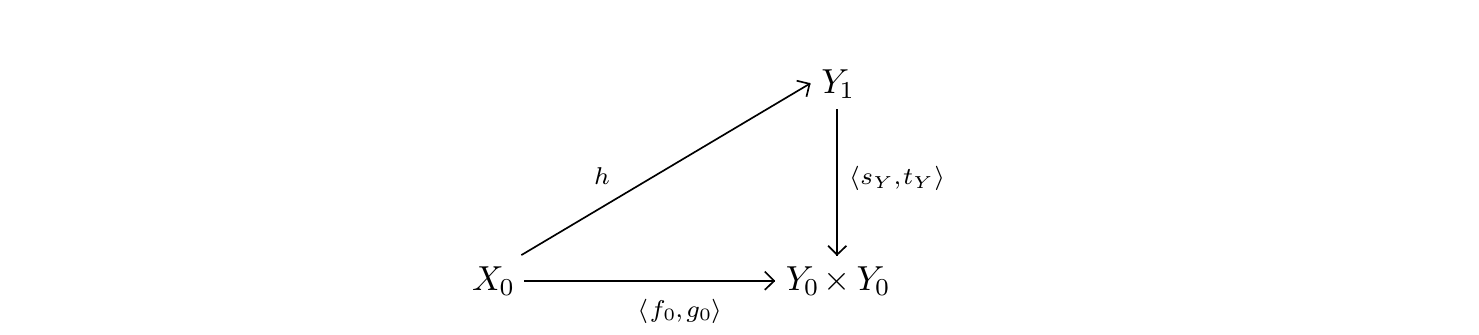}}
    \end{center}

    with respect to $\pr{s_Y}{t_Y}$.
  \end{enumeratenumeric}
\end{definition}

\begin{notation}
  We shall write $h : f \Rightarrow g$ for a homotopy $h$.
\end{notation}

\begin{remark}
  \label{rem:edge}Let $f, g : X \rightarrow Y$ be quiver morphisms. A homotopy
  $h : f \Rightarrow g$ is just a family
  \[  \left( h (x) : f_0 (x) \rightsquigarrow g_0 (x) \right)_{x \in X_0} \]
  of edges of $Y_1$. Notice that for any such edge there is an edge $\sigma_Y
  (h (x)) : g_0 (x) \rightsquigarrow f_0 (x)$ in the opposite direction.
\end{remark}

Suppose now $\mathbb{C}$ is lex.

\begin{definition}
  A {\tmem{pseudo-groupoid}} $X$ in $\mathbb{C}$ is a quiver $s_X, t_X : X_1
  \rightrightarrows X_0$ equipped with structural morphisms
  \begin{itemizeminus}
    \item $\rho_X : X_0 \rightarrow X_1$ such that
    \begin{enumerateroman}
      \item $s_X \circ \rho_X = \tmop{id}$;

      \item $t_X \circ \rho_X = \tmop{id}$;
    \end{enumerateroman}
    \item $\sigma_X : X_1 \rightarrow X_1$ such that
    \begin{enumerateroman}
      \item $s_X \circ \sigma_X = t_X$;

      \item $t_X \circ \sigma_X = s_X$;
    \end{enumerateroman}
    \item $\tau_X : X_1 \circledast X_1 \rightarrow X_1$ with $X_1
    \circledast X_1$ given by pullback of $s_X$ along $t_X$, such that
    \begin{enumerateroman}
      \item $s_X \circ \tau_X = s_X \circ p_0$;

      \item $t_X \circ \tau_X = t_X \circ p_1$.
    \end{enumerateroman}
  \end{itemizeminus}
\end{definition}

Entities we call pseudo-groupoids are also (and perhaps better) known as
{\tmem{pseudo-equivalence relations}}
{\cite{carboni1995some,carboni1982free}}. As the only pseudo-groupoids of
interest here are those internal to a lex category $\mathbb{C}$, we shall once
again dispense with the qualifier ``in $\mathbb{C}$''.

Notice that $\rho_X$ is a local section of $\pr{s_X}{t_X} : X_1 \rightarrow
X_0 \times X_0$ with respect to the diagonal. The homotopy relation is stable
by composition while composition of quiver morphisms among pseudo-groupoids is
associative up-to homotopy, so pseudo-groupoids organise themselves in the
category $\qgrpd{\mathbb{C}}$ along with homotopy classes of quiver morphisms
among them. We have
\begin{eqnarray*}
  \exlex{\mathbb{C}} & \simeq & \qgrpd{\mathbb{C}}
\end{eqnarray*}
{\cite{carboni1995some,carboni1982free}}. Quite unsurprisingly, the functor
$\tmmathbf{y} : \mathbb{C} \rightarrow \exlex{\mathbb{C}}$ turns then out to
be the one sending an object $X \in \mathbb{C}$ on the {\tmem{discrete}}
pseudo-groupoid $\tmop{id}, \tmop{id} : X \rightrightarrows X$. As pointed out
in {\cite{shulman2021derivator}}, pseudo-groupoids give rise to ``groupoidal''
bicategories internal to $\mathbb{C}$, with homotopies as 2-cells. So if we
restrain from quotienting by the homotopy relation, we end up with a
tricategory where these data organise themselves (taking into account
coherence conditions). The homotopy 1-category of this tricategory is then
$\qgrpd{\mathbb{C}}$ (see also {\cite{kinoshita2014category}}).

\section{Exact Completions of Implicative Assemblies}
\label{sec:exact-completion-asm}

\subsection{The Exact Completion of
$\asma$}

In this subsection we address the ex/reg completion of
$\asma$ as well as the ex/lex
completion of $\mass$, exploiting concepts and
results found in {\cite{menni2000exact}}.

\begin{definition}
  \label{def:chaos}{\tmdummy}

  \begin{enumeratenumeric}
    \item A {\tmem{chaotic situation}} is given by two lex categories
    $\mathbf{S}$ and $\mathbb{C}$ connected by an adjunction $| - | \dashv
    \Delta : \mathbf{S} \rightarrow \mathbb{C}$.

    \item A morphism $f : X \rightarrow Y$ in $\mathbb{C}$ is a
    {\tmem{pre-embedding}} if the square

    \begin{center}
      \raisebox{-0.5\height}{\includegraphics[width=14.8909550045914cm,height=3.4708448117539cm]{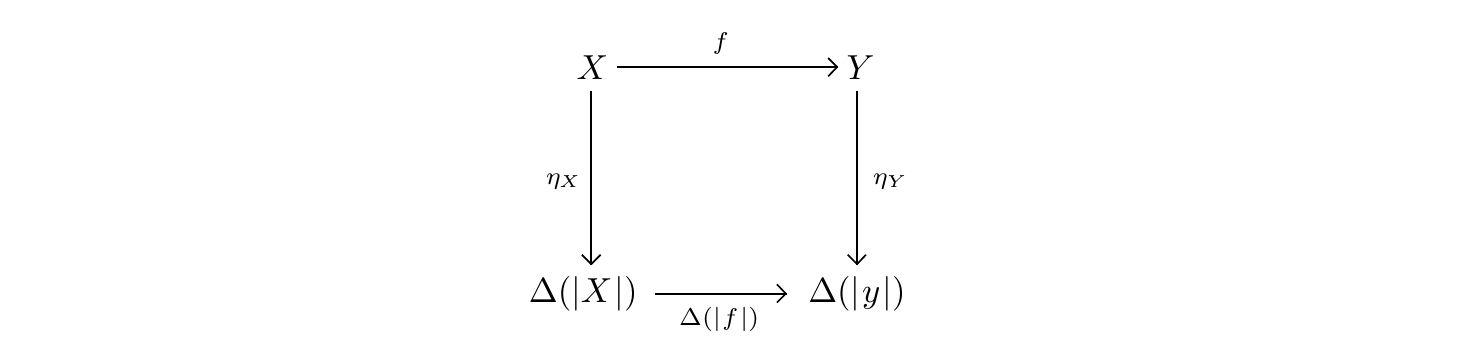}}
    \end{center}

    is a pullback.

    \item An object $\Upsilon \in \mathbb{C}$ is {\tmem{generic}} if for every
    $X \in \mathbb{C}$ there is a pre-embedding $X \rightarrow \Upsilon$.
  \end{enumeratenumeric}
\end{definition}

(as a word of caution, we use the symbol $\Delta$ for $\nabla$ in op.cit.). A
generic object in the above sense is not the same as (although related to) the
homonymous fibred artefact. The notion of chaotic situation is an abstract
version of what happens when we have a forgetful functor from a concrete
category into $\ens$.

\begin{proposition}[Menni]
  \label{prop:menni}Let $\mathbb{C}$ have a chaotic situation. The following
  are equivalent
  \begin{enumerateroman}
    \item $\mathbb{C}$ has a generic object;

    \item $\mathbb{C}$ has a generic mono.
  \end{enumerateroman}
\end{proposition}

This is Prop 8.1.8 in op.cit.

\begin{theorem}[Menni]
  \label{th:menni}Let $\mathbb{D}$ be a regular LCCC with generic mono. Then
  $\mathbb{D}_{\tmop{ex} / \tmop{reg}}$ is a topos.
\end{theorem}

This is Theorem 11.3.3 in op.cit.

\begin{remark}
  \label{rem:chaotic}$(\mathcal{S}, \tmop{id}_{\mathcal{S}})$ is a generic
  object in the sense of Definition \ref{def:chaos}.
\end{remark}

\begin{theorem}
  $(\asma)_{\tmop{ex} / \tmop{reg}}$ is
  a topos.
\end{theorem}

\begin{proof}
  Proposition \ref{prop:menni}, Theorem \ref{th:menni} and Remark
  \ref{rem:chaotic}.
\end{proof}

\begin{corollary}
  \label{cor:ex-lex}Let $M$ be a generator. Then
  \[ \left( \mass \right)_{\tmop{ex} / \tmop{lex}} \simeq \left( \left(
     \mass \right)_{\tmop{reg} / \tmop{lex}} \right)_{\tmop{ex} /
     \tmop{reg}} \simeq
     (\asma)_{\tmop{ex} / \tmop{reg}}
  \]
  is a topos.
\end{corollary}

\subsection{Tripos-to-Topos Construction}

\begin{definition}
  An {\tmem{implicative set}} $\eset{X}$ is a set equipped with a
  {\tmem{non-standard equality}}, that is a predicate $| - \approx - | : X
  \times X \rightarrow \mathcal{A}$ which is
  \begin{enumerateroman}
    \item {\tmem{symmetric:}} $\vld{\bigM_{x, x' \in X} \left(
    \eeq{x}{x'} \rightarrow \eeq{x'}{x} \right)}$;

    \item {\tmem{transitive:}} $ \vld{\bigM_{x, x', x'' \in X} \left(
    \eeq{x}{x'} \sqcap \text{ }  \eeq{x'}{x''} \rightarrow
    \eeq{x}{x''} \right)}$.
  \end{enumerateroman}
\end{definition}

Implicative sets are the objects of $\ens [\mathcal{A}]$. Notice that we do
not necessarily have {\tmem{reflexivity}}, that is $\eeq{x}{x} \in
\mathcal{S}$ for all $x \in X$. This is in a sense the main point of the
construction, as $\eeq{x}{x}$ is interpreted as an {\tmem{existence
predicate}}. Accordingly, we call {\tmem{ghosts}} elements $x \in X$ such that
$\eeq{x}{x} \nin \mathcal{S}$.

\begin{notation}
  {\tmdummy}

  \begin{itemizeminus}
    \item $\sym{X} \assign \bigM_{x, x' \in X} \left( \eeq{x}{x'}
    \rightarrow \eeq{x'}{x} \right)$;

    \item $\trans{X} \assign \bigM_{x, x', x'' \in X} \left(
    \eeq{x}{x'} \sqcap \text{ }  \eeq{x'}{x''} \rightarrow
    \eeq{x}{x''} \right)$;

    \item $\eexi{X} \assign \eeq{x}{x}$.
  \end{itemizeminus}
\end{notation}

\begin{definition}
  Let $\eset{X}$ and $\eset{Y}$ be implicative sets. A $\tmop{functional}
  \tmop{relation}$
  \[ F : \eset{X} \rightarrow \eset{Y} \]
  is a predicate $F : X \times Y \rightarrow \mathcal{A}$ which is
  \begin{enumerateroman}
    \item {\tmem{extensional:}} $ \vld{\bigM_{x, x' \in X}
    \bigM_{y, y' \in Y} \left( F (x, y) \sqcap \eeq{x}{x'} \sqcap
    \eeq{y}{y'} \rightarrow F (x', y') \right)}$;

    \item {\tmem{strict:}}
    $\vld{
    \bigM_{x \in X} \bigM_{y \in Y} \left( F(x, y)
    \rightarrow \eexi{X}(x) \sqcap \eexi{Y}(y) \right)
    }$;

    \item {\tmem{single-valued:}} $\vld{\bigM_{x \in X}
    \bigM_{y, y' \in Y} \left( F (x, y) \sqcap F (x, y') \rightarrow
    \eeq{y}{y'} \right)}$;

    \item {\tmem{total:}} $ \vld{\bigM_{x \in X} \left( \eexi{X} (x)
    \rightarrow \iex{y \in Y_0} \left( \eexi{Y} (y) \sqcap F (x, y) \right)
    \right)}$.
  \end{enumerateroman}
  Two functional relations $F, G : \eset{X} \rightarrow \eset{Y}$ are
  {\tmem{equivalent}} if they are isomorphic in $\mathbb{P} [\mathcal{A}]_{X
  \times Y}$.
\end{definition}

\begin{notation}
  {\tmdummy}

  \begin{itemizeminus}
    \item $\ext{F} \assign \bigM_{x, x' \in X} \bigM_{y, y'
    \in Y} \left( F (x, y) \sqcap \eeq{x}{x'} \sqcap \eeq{y}{y'}
    \rightarrow F (x', y') \right)$;

    \item $\str{F} \assign \bigM_{x \in X} \bigM_{y \in Y}
    \left( F (x, y) \rightarrow \eexi{X} (x) \sqcap \eexi{Y} (y) \right)$;

    \item $\sv{F} \assign \bigM_{x \in X} \bigM_{y, y' \in
    Y} \left( F (x, y) \sqcap F (x, y') \rightarrow \eeq{y}{y'} \right)$;

    \item $\tot{F} \assign \bigM_{x \in X} \left( \eexi{X} (x)
    \rightarrow \iex{y \in Y_0} \left( \eexi{Y} (y) \sqcap F (x, y) \right)
    \right)$.
  \end{itemizeminus}
\end{notation}

\begin{proposition}
  \label{prop:topos-morphisms}Two functional relations $f, g : \eset{X}
  \rightarrow \eset{Y}$ are equivalent if and only if
  \[ \vld{\bigM_{x \in X, y \in Y} (F (x, y) \rightarrow F (x, y))}
  \]
\end{proposition}

\subsection{The Exact Completion of $\mass$ in
Context}

\

Let $M \subseteq \sepa$ be a generator. We already observed that $\left(
\mass \right)_{\tmop{ex} / \tmop{lex}}$ is a topos (c.f. Corollary
\ref{cor:ex-lex}), so the question is now if it is the same topos as $\ens
[\mathcal{A}]$.

\begin{notation}
  In what follows we make systematic use $\lambda$-terms, some of them with
  relatively long arguments lists containing dummy arguments. We shall use the
  symbol {\dum} for dummies in order to improve readability.
\end{notation}

\subsubsection{A fully faithful functor $\mathcal{K}$.}

\begin{remark}
  \label{rem:ex-lex-val}A pseudo-groupoid $X$ induces a valuation
  \begin{eqnarray*}
    \nu_X : X_0 \times X_0 & \longrightarrow & \ppows{M}\\
    (x, x') & \mapsto & \left\{ \exi{X_1} (e) |e \in X (x, x') \right\}
  \end{eqnarray*}
\end{remark}

\begin{lemma}
  \label{lem:impl-set}Let $X$ be a pseudo-groupoid. $X_0$ equipped with the
  relation
  \begin{eqnarray*}
    \ieeq{x}{x'}{X} & \assign & \exi{X_0} (x) \sqcap \exi{X_0} (x') \sqcap
    \siex \nu_X (x, x')
  \end{eqnarray*}
  is an implicative set.
\end{lemma}

\begin{proof}
  We have
  \begin{eqnarray*}
    \sym{X} & = & \bigM_{x, x' \in X_0} \left( \exi{X_0}
    (x) \sqcap \exi{X_0} (x') \sqcap \siex \nu_X (x, x') \rightarrow \exi{X_0}
    (x') \sqcap \exi{X_0} (x) \sqcap \siex \nu_X (x', x) \right)
  \end{eqnarray*}
  Let $e : x \rightsquigarrow x' \in X_1$ be an edge. We then have the
  following subquiver of $X$

  \begin{center}
    \raisebox{-0.5\height}{\includegraphics[width=14.8909550045914cm,height=1.78518955791683cm]{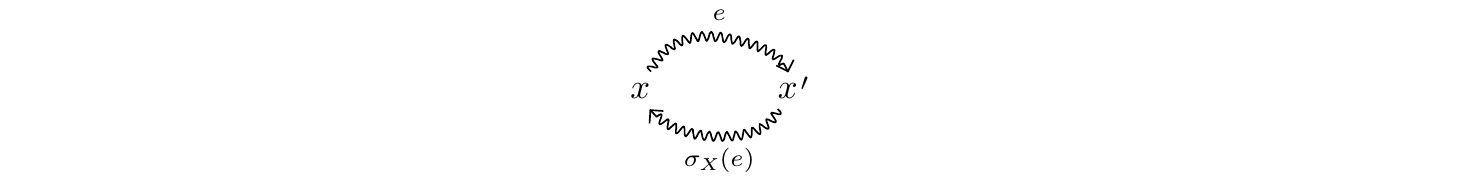}}
  \end{center}

  \

  with the ``opposite'' edge $\sigma_X (e) : x' \rightsquigarrow x$. Let
  $\xi$ be a tracker of $\sigma_X$. We then have
  \begin{eqnarray*}
    \lam{u \nocomma v \nocomma w}{} u \sqcap v \sqcap \intpr{\lam{z}{z \left(
    \xi \left( w \ide \right) \right)}} & \preccurlyeq & \sym{X}
  \end{eqnarray*}
  and therefore $\vld{\sym{X}}$. Next we have
  \begin{eqnarray*}
    \trans{X} & = & \bigM_{x, x', x'' \in X_0} (\\
    &  & \text{\qquad} \exi{X_0} (x) \sqcap \exi{X_0} (x') \sqcap \siex \nu_X
    (x, x') \sqcap\\
    &  & \text{\qquad} \exi{X_0} (x') \sqcap \exi{X_0} (x'') \sqcap \siex
    \nu_X (x', x'')\\
    &  & \text{\qquad} \rightarrow\\
    &  & \text{\qquad} \text{$\exi{X_0} (x) \sqcap \exi{X_0} (x'') \sqcap
    \siex \nu_X (x, x'')$}\\
    &  & \left. \text{ } \right)
  \end{eqnarray*}
  Let $e_1 : x \rightsquigarrow x'$ and $e_2 : x' \rightsquigarrow x''$ be
  edges. We then have the following subquiver of $X$

  \begin{center}
    \raisebox{-0.50000516673039\height}{\includegraphics[width=14.8909550045914cm,height=1.58685884822248cm]{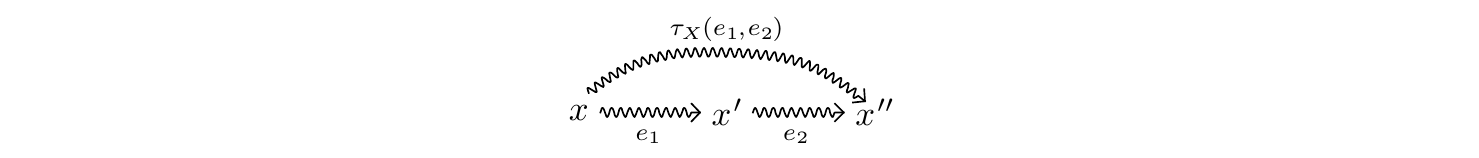}}
  \end{center}

  with the ``composed'' edge $\tau_X (e_1, e_2) : x \rightsquigarrow x''$. Let
  $\chi$ be a tracker of $\tau_X$. We then have
  \begin{eqnarray*}
    \lam{u \dum w \dum v w'}{u \sqcap v \sqcap \intpr{\lam{z}{\chi \left(
    \left( w \ide \right) \sqcap \left( w'  \ide \right) \right)}}} &
    \preccurlyeq & \trans{X}
  \end{eqnarray*}
  and therefore $\vld{\trans{X}}$.
\end{proof}

Let $f : X \rightarrow Y$ be a morphism of pseudo-groupoids. The relation
$\Rf$ from $(X, \approx_X)$ to $(Y, \approx_Y)$ is given by
\begin{eqnarray*}
  \Rf (x, y) & \assign & \exi{X_0} (x) \sqcap \exi{Y_0} (y) \sqcap
  \siex \nu_Y (y, f_0 (x))
\end{eqnarray*}
\begin{remark}
  \label{rem:rf-values}Let $(x, y) \in X \times Y$. We have
  $\vld{\Rf (x, y)}$ if and only if there is an edge $e : y'
  \rightsquigarrow f_0 (x)$ in $Y$ and $\Rf (x, y) = \bot$
  otherwise.
\end{remark}

\begin{remark}
  \label{rem:rf-implies}Let $\mathsf{Rf} \assign \bigM_{x \in X_0, y
  \in Y_0} \left( \Rf (x, y) \rightarrow \ieeq{y}{f_0 (x)}{Y}
  \right)$. We have
  \begin{eqnarray*}
    \mathsf{Rf} & = & \bigM_{x \in X_0, y \in Y_0} (\\
    &  & \text{\qquad} \exi{X_0} (x) \sqcap \exi{Y_0} (y) \sqcap \siex \nu_Y
    (y, f_0 (x))\\
    &  & \text{\qquad} \rightarrow\\
    &  & \text{\qquad} \exi{Y_0} (y) \sqcap \exi{Y_0} (f_0 (x)) \sqcap \siex
    \nu_Y (y, f_0 (x))\\
    &  & \left. \text{ } \right)
  \end{eqnarray*}
  Suppose $\xi$ tracks $f_0$. Then
  \begin{eqnarray*}
    \intpr{\lam{u \nocomma v \nocomma w}{v \sqcap (\xi u) \sqcap} w} &
    \preccurlyeq & \mathsf{Rf}
  \end{eqnarray*}
  and therefore $\vld{Rf}$.
\end{remark}

\begin{lemma}
  \label{lem:ext}$\Rf$ is extensional.
\end{lemma}

\begin{proof}
  We have
  \begin{eqnarray*}
    \ext{\Rf} & = & \bigM_{x, x' \in X_0}
    \bigM_{y, y' \in Y_0} (\\
    &  & \text{{\hspace{3em}}$\exi{X_0} (x) \sqcap$} \exi{Y_0} (y) \sqcap
    \siex \nu_Y (y, f_0 (x)) \sqcap\\
    &  & \text{{\hspace{3em}}} \exi{X_0} (x) \sqcap \exi{X_0} (x') \sqcap
    \siex \nu_X (x, x') \sqcap\\
    &  & \text{{\hspace{3em}}} \text{} \exi{Y_0} (y) \sqcap \exi{Y_0} (y')
    \sqcap \siex \nu_Y (y, y')\\
    &  & \text{{\hspace{3em}}} \rightarrow\\
    &  & \text{{\hspace{3em}}} \text{$\exi{X_0} (x') \sqcap$} \exi{Y_0} (y')
    \sqcap \siex \nu_Y (y', f_0 (x'))\\
    &  & \left. \text{ } \right)
  \end{eqnarray*}
  Let $x, x'' \in X_0$ be vertices of $X$ and $y, y' \in Y_0$ be vertices of
  $Y$. Let $e : x \rightsquigarrow x'$ be an edge of $X$. Let $e_1 : y
  \rightsquigarrow f_0 (x)$ and $e_2 : y \rightsquigarrow y'$ be edges of $Y$
  (notice that besides their ``types'' we did not make any hypotheses on the
  vertices and edges in question). We then have the following subquiver of $Y$

  \begin{center}
    \raisebox{-0.5\height}{\includegraphics[width=14.8909550045914cm,height=4.26410205955661cm]{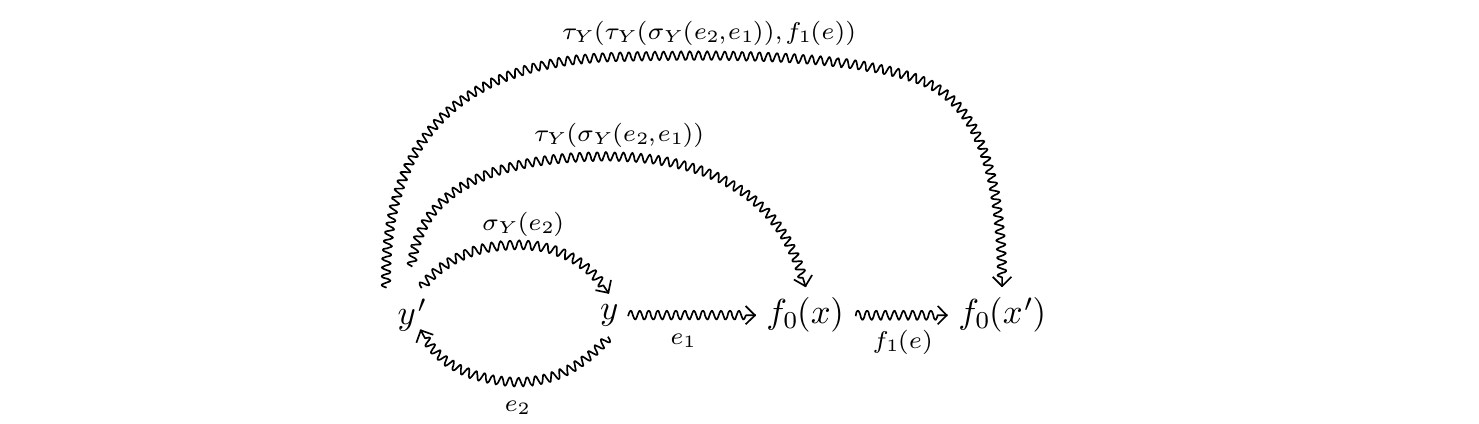}}
  \end{center}

  \

  Suppose $\chi$ tracks $f_1$, $\kappa$ tracks $\sigma_Y$ and $\varpi$ tracks
  $\tau_Y$. Then
  \begin{eqnarray*}
    \intpr{\lam{\dum  \dum w \dum v w'  \dum v' w''}{} v \sqcap v' \sqcap
    \lam{z}{z \left( \varpi \left( \left( \varpi \left( \kappa \left( w''
    \ide \right) \right) \sqcap \left( w \ide \right) \sqcap \left( \chi
    \left( w'  \ide \right) \right) \right) \right) \right)}} & \preccurlyeq &
    \ext{\Rf}
  \end{eqnarray*}
  and therefore $\vld{\ext{\Rf}}$.
\end{proof}

\begin{lemma}
  \label{lem:str}$\Rf$ is strict.\quad
\end{lemma}

\begin{proof}
  We have
  \begin{eqnarray*}
    \str{\Rf} & = & \bigM_{x \in X_0} \bigM_{y
    \in Y_0} (\\
    &  & \text{\qquad$\exi{X_0} (x) \sqcap$} \exi{Y_0} (y) \sqcap \siex \nu_Y
    (y, f_0 (x))\\
    &  & \text{\qquad} \rightarrow\\
    &  & \text{\qquad} \exi{X_0} (x) \sqcap \siex \nu_X (x, x) \sqcap\\
    &  & \text{\qquad} \exi{Y_0} (y) \sqcap \siex \nu_Y (y, y)\\
    &  & \left. \text{ } \right)
  \end{eqnarray*}
  Suppose $\kappa$ tracks $\rho_X$ while $\kappa'$ tracks $\rho_Y$. Then
  \begin{eqnarray*}
    \intpr{\lam{u v \dum}{u \sqcap \lam{z}{z (\kappa \nocomma u)} \sqcap v
    \sqcap \lam{z}{z (\kappa' \nocomma v)}}} & \preccurlyeq &
    \str{\Rf}
  \end{eqnarray*}
  and therefore $\vld{\str{\Rf}}$.
\end{proof}

\begin{lemma}
  \label{lem:sv}$\Rf$ is single valued.
\end{lemma}

\begin{proof}
  We have
  \begin{eqnarray*}
    \sv{\Rf} & = & \bigM_{x \in X_0} \bigM_{y, y'
    \in Y_0} (\\
    &  & \text{\qquad} \text{$\exi{X_0} (x) \sqcap$} \exi{Y_0} (y) \sqcap
    \siex \nu_Y (y, f_0 (x)) \sqcap\\
    &  & \text{\qquad} \text{$\exi{X_0} (x) \sqcap$} \exi{Y_0} (y') \sqcap
    \siex \nu_Y (y', f_0 (x))\\
    &  & \text{\qquad} \rightarrow\\
    &  & \text{\qquad} \exi{Y_0} (y) \sqcap \exi{Y_0} (y') \sqcap \siex \nu_Y
    (y, y')\\
    &  & \left. \text{ } \right)
  \end{eqnarray*}
  Let $e : y \rightsquigarrow f_0 (x)$ and $e' : y' \rightsquigarrow f_0 (x)$
  be edges of $Y$. We then have the subquiver of $Y$

  \begin{center}
    \raisebox{-0.5\height}{\includegraphics[width=14.8909550045914cm,height=2.28095894004985cm]{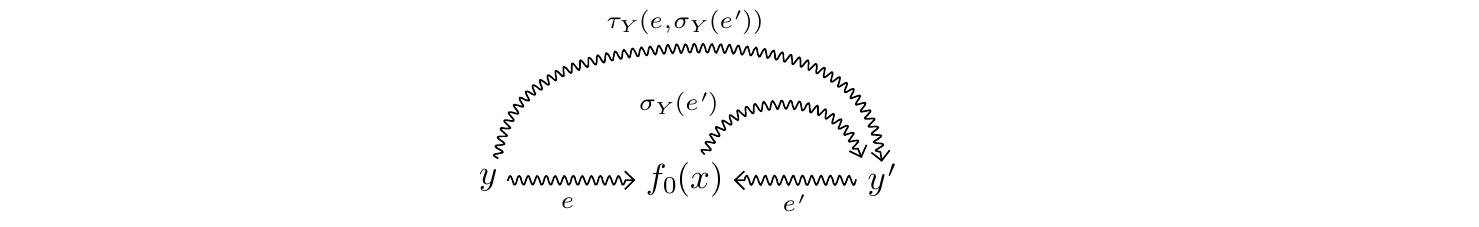}}
  \end{center}

  Suppose $\xi$ tracks $\sigma_Y$ and $\varpi$ tracks $\tau_Y$. Then
  \begin{eqnarray*}
    \intpr{\lam{\dum v w \dum v' w'}{v \sqcap v' \sqcap \lam{z}{z \left(
    \varpi \left( \left( w \ide \right) \sqcap \left( \xi \left( w'  \ide
    \right) \right) \right) \right)}}} & \preccurlyeq & \sv{\Rf}
  \end{eqnarray*}
  and therefore $\vld{\sv{\Rf}}$.
\end{proof}

\begin{lemma}
  \label{lem:tot}$\Rf$ is total.
\end{lemma}

\begin{proof}
  We have
  \begin{eqnarray*}
    \tot{\Rf} & = & \bigM_{x \in X_0} (\\
    &  & \text{\qquad} \exi{X_0} (x) \sqcap \siex \nu_X (x, x)\\
    &  & \text{\qquad} \rightarrow\\
    &  & \text{\qquad} \iex{y \in Y_0} \left( \exi{Y_0} (y) \sqcap \siex
    \nu_Y (y, y) \sqcap \exi{X_0} (x) \sqcap \siex \nu_Y (y, f_0 (x))
    \right)\\
    &  & \left. \text{ } \right)
  \end{eqnarray*}
  Suppose $\chi$ tracks $f_0$ while $\kappa$ tracks $\rho_Y$. We know $y \in
  Y_0$ that fits, namely $f_0 (x)$. Then
  \begin{eqnarray*}
    \intpr{\lam{u \nocomma v}{(\chi u) \sqcap} \lam{z}{\kappa (\chi u)} \sqcap
    u \sqcap \lam{z}{\kappa (\chi u)}} & \preccurlyeq & \tot{\Rf}
  \end{eqnarray*}
  and therefore $\vld{\tot{\Rf}}$.
\end{proof}

\begin{proposition}
  \label{prop:K}The assignment
  \begin{eqnarray*}
    \mathcal{K}: \left( \mass \right)_{\tmop{ex} / \tmop{lex}} &
    \longrightarrow & \ens [\mathcal{A}]\\
    X & \mapsto & (X_0, \approx_X)
  \end{eqnarray*}
  extends to a faithful functor with action on morphisms $f : X \rightarrow Y$
  given by
  \begin{eqnarray*}
    \mathcal{K} (f) (x, y) & \assign & \Rf (x, y)
  \end{eqnarray*}
\end{proposition}

\begin{proof}
  $(X_0, \approx_X)$ is an implicative set by Lemma \ref{lem:impl-set} while
  $\Rf$ is a functional relation by Lemmas \ref{lem:ext},
  \ref{lem:str}, \ref{lem:sv} and \ref{lem:tot}. Let $g : X \rightarrow Y$ be
  a further morphism of pseudo-groupoids and $h : f \Rightarrow g$ be a
  homotopy. There is thus an edge $h (x) : f_0 (x) \rightsquigarrow g_0 (x)$
  for every vertex $x \in X_0$ (c.f. Remark \ref{rem:edge}), which entails
  $\vld{\siex \nu_Y (f_0 (x), g_0 (x))}$ and therefore
  \[ \vld{\ieeq{f_0 (x)}{g_0 (x)}{Y}} \]
  We further have $\vld{\bigM_{x \in X_0} \bigM_{y \in Y_0}
  \left( \Rf (x, y) \rightarrow \ieeq{y}{f_0 (x)}{Y} \right)}$
  (c.f. Remark \ref{rem:rf-implies}). Now $\Rf (x, y)$ is either
  $\bot$ (so it implies anything) or $\vld{\Rf (x, y)}$ (c.f.
  Remark \ref{rem:rf-values}). In the latter case we have
  $\vld{\ieeq{y}{f_0(x)}{Y}}$
  by (the implicative) modus ponens and further
  $\vld{\ieeq{y}{g_0(x)}{Y}}$ by symmetry and transitivity, hence
  \[ \vld{\Rf \rightarrow \mathfrak{R}_g} \]
  which entails that $\mathcal{K}$ is well-defined on morphisms (c.f.
  Proposition \ref{prop:topos-morphisms}). It is easy to see that
  $\mathcal{K}$ preserves identities and composition.

  \

  Let $u, v : X \rightarrow Y$ be morphisms of pseudo-groupoids. Suppose
  $\mathcal{K} (u) =\mathcal{K} (v)$, that is
  \[ \vld{\mathfrak{R}_u \leftrightarrow \mathfrak{R}_v} \quad (\star) \]
  Given $x \in X_0$ $(\star)$ implies $\vld{\ieeq{u_0 (x)}{v_0 (x)}{Y_0}}$, so
  there is an edge $e_x : u_0 (x) \rightsquigarrow v_0 (x)$. All these edges
  are indexed by $X$ and organise themselves in a homotopy $h : u \Rightarrow
  v$, hence $\mathcal{K}$ is faithful.
\end{proof}

\begin{notation}
  Still seeking the greater good by making things human-readable, we shall
  also add the $\letin{x}{e}$ macro to our $\lambda$-calculus.
\end{notation}

\begin{proposition}
  \label{prop:full}The functor $\mathcal{K}:
  \exlex{(\mass)} \rightarrow \ens [\mathcal{A}]$
  is full.
\end{proposition}

\begin{proof}
  Let $X, Y \in \exlex{(\mass)}$ and $F :
  \mathcal{K} (X) \rightarrow \mathcal{K} (Y)$ be a functional relation.
  Unravelling the definition yields in this case
  \begin{eqnarray*}
    \vld{\sv{F}} & = & \bigM_{x \in X_0} \bigM_{y, y' \in
    Y_0} \left( F (x, y) \sqcap F (x, y') \rightarrow \exi{Y_0} (y) \sqcap
    \exi{Y_0} (y') \sqcap \siex \nu_Y (y, y')  \right)\\
    \vld{\tot{F}} & = & \bigM_{x \in X_0} \left( \left( \exi{X_0} (x)
    \sqcap \siex \nu_X (x, x) \right) \rightarrow \iex{y \in Y_0} \left(
    \exi{Y_0} (y) \sqcap \siex \nu_Y (y, y) \sqcap F (x, y) \right) \right)
  \end{eqnarray*}
  As implicative sets in the image of $\mathcal{K}$ do not have ghosts, $F$
  can be presented as a map
  \begin{eqnarray*}
    \tilde{F} : X_0 & \longrightarrow & \pows{Y_0}\\
    x & \mapsto & \left\{ y \in Y_0 |F (x, y) \in \sepa \right\}
  \end{eqnarray*}
  Assuming choice we can extract a map $f_0 : X_0 \rightarrow Y_0$ from
  $\tilde{F}$. Suppose $\kappa$ tracks $\rho_X$. For any $x \in X_0$ there is
  $y_x$ (in general not unique) such that
  \begin{eqnarray*}
    \letin{t}{\lam{p}{\tot{F}  (p \sqcap (\kappa p))} } &  & \\
    \intpr{\lam{q}{\lpair{\lpizero}{\lpitwo} \left( (t q)  \ide \right)}}  &
    \preccurlyeq & \bigM_{x \in X_0} \left( \exi{X_0} (x) \rightarrow
    \left(  \exi{Y_0} (y_x) \sqcap F (x, y_x) \right) \right)
  \end{eqnarray*}
  thus
  \[ \vld{} \bigM_{x \in X_0} \left( \exi{X_0} (x) \rightarrow \left(
     \exi{Y_0} (y_x) \sqcap F (x, y_x) \right) \right) \]
  But $\vld{F (x, f_0 (x))}$ for all $x \in X$ by definition of $f_0$, hence
  \begin{eqnarray*}
    \letin{t'}{\lam{r}{\sv{F}  \left( \left( \lpione r \right) \sqcap F (x, f_0
    (x))  \right)}} &  & \\
    \ilam{s}{\lpione  (t' s)} & \preccurlyeq & \bigM_{x \in X_0}
    \left(  \exi{Y_0} (y_x) \sqcap F (x, y_x) \right) \rightarrow \exi{Y_0}
    (f_0 (x))
  \end{eqnarray*}
  thus
  \[ \vld{} \bigM_{x \in X_0} \left(  \exi{Y_0} (y_x) \sqcap F (x,
     y_x) \right) \rightarrow \exi{Y_0} (f_0 (x)) \]

  But then
  \[ \vld{} \bigM_{x \in X_0} \left( \exi{X_0} (x) \rightarrow
     \exi{Y_0} (f_0 (x)) \right) \]
  so $f_0$ is tracked.

  \

  Let $e : x \rightsquigarrow x'$ be an edge of $X$, so we have
  $\vld{\ieeq{x}{x'}{X}}$. We have $\vld{F (x, f_0 (x))}$ and $\vld{F (x', f_0
  (x'))}$ by definition of $f_0$, hence $\vld{F (x, f_0 (x'))}$ by
  extensionality and therefore $\vld{} \ieeq{f_0 (x)}{f_0 (x')}{Y}$ by
  single-valuedness. But then $Y (f_0 (x), f_0 (x')) \neq \varnothing$, so
  given $e : x \rightsquigarrow x'$ we have
  \begin{eqnarray*}
    \left( (f_0 \times f_0) \circ \pr{s_X}{t_X} \right) (e) & = & \left(
    \pr{s_Y}{t_Y} \circ f_1 \right) (e)
  \end{eqnarray*}
  for any map such that
  \begin{eqnarray*}
    f_1 (e) & \in & Y (f_0 (x), f_0 (x'))
  \end{eqnarray*}
  Assuming choice such a map always exists, what remains to be shown is that
  it is tracked. Consider the map
  \begin{eqnarray*}
    \phi : \mathcal{A} & \longrightarrow & \mathcal{A}\\
    a & \mapsto & \left\{\begin{array}{lll}
      \exi{Y_1} (f_1 (e)) &  & \text{if there are } x, x' \in X_0  \text{and}
      e \in X (x, x')  \text{such that} a = \exi{X_1} (e)\\
      \top &  & \text{otherwise}
    \end{array}\right.
  \end{eqnarray*}
  We have
  \begin{eqnarray*}
    \lamm{\phi} & = & \bigM_{a \in \mathcal{A}} (a \rightarrow \phi
    (a))
  \end{eqnarray*}
  Notice that $\lamm{\phi} \in \sepa$ as either $a \in \sepa$ and $\phi (a)
  \in \sepa$ or $\phi (a) = \top$. Furthermore, we have
  \begin{eqnarray*}
    \ilam{p}{\left( \lamm{\phi} \right) p} & \preccurlyeq & \bigM_{e
    \in X_1} \left( \exi{X_1} (e) \rightarrow \exi{Y_1} (f_1 (e)) \right)
  \end{eqnarray*}
  hence
  \[ \vld{} \bigM_{e \in X_1} \left( \exi{X_1} (e) \rightarrow
     \exi{Y_1} (f_1 (e)) \right) \]

\end{proof}

\subsubsection{Is $\mathcal{K}$ always essentially surjective?}
\mbox{}\\ \\
Let $\eset{X}$ be an implicative set and $\nu : X \times X \rightarrow
\ppows{M}$ be a valuation such that $\approx^+ \text{ } \cong_{X \times X}
\text{ } \siex \nu$ (such a valuation always exists since $M$ is dense).
Assuming choice there is a map $c : \ngh{X} \rightarrow M$ such that $c (x)
\in \nu (x, x)$ for all $x \in \ngh{X}$.

\begin{notation}
  \leavevmode

  \begin{enumeratenumeric}
    \item $\Equ{X} \assign \left\{ (x, x') \in X \times X| \text{ }
    \eeq{x}{x'} \in \sepa \right\}$ (the set of valid equalitites on $X$);

    \item $\ngh{X} \assign \left\{ x \in X| \eexi{X} (x) \in \sepa \right\}$
    (the set of $X$'s non-ghosts);

    \item ${\approx}^{+} \assign \restr{\approx}{\Equ{X}}$;

    \item $\textsf{} \exiplus{X} \assign \restr{\left( \eexi{X}
    \right)}{\ngh{X}}$.
  \end{enumeratenumeric}
\end{notation}

\begin{lemma}
  \label{lem:pseudo}The $M$-assemblies
  \begin{itemizeminus}
    \item $\hat{X}_1$ given by $\car{\widehat{X_1}} \assign \sum_{(x, x') \in
    \Equ{X}} \nu (x, x')$ and $\exi{\hat{X}_1} ((x, x'), m) \assign c (x)
    \sqcap m \sqcap c (x')$;

    \item $\hat{X}_0$ given by $\car{\hat{X}_0} \assign \ngh{X}$ and
    $\exi{\widehat{X_0}} (x) = c (x)$
  \end{itemizeminus}
  along with the maps $s \assign \pi_0 \circ \pi_0$ and $t \assign \pi_1 \circ
  \pi_0$ form a pseudo-groupoid $\hat{X}$.
\end{lemma}

\begin{proof}
  $\widehat{X_1}$ is an $M$-assembly since $M$ is algebraic. The source and
  target maps $s$ and $t$ are tracked by $\lpizero$ and $\lpitwo$,
  respectively. We have the local section of $\pr{s}{t}$
  \begin{eqnarray*}
    \rho : & \car{\hat{X}_0} \longrightarrow & \car{\widehat{X_1}}\\
    x & \mapsto & ((x, x), c (x))
  \end{eqnarray*}
  tracked by $\tmop{id}$. Next we have the iso
  \[ \approx^+ \text{ } \cong_{X \times X} \text{ } \siex \nu \qquad (\star)
  \]
  by density, which entails
  \[ \vld{\sym{X}' \assign \bigM_{x, x' \in \ngh{X}} \left( \siex \nu
     (x, x') \rightarrow \siex \nu (x', x) \right)} \]
  Let $x, x' \in \ngh{X}$ and $m \in \nu (x, x')$. We then have $\left(
  \sym{X}' \ilam{z}{z m} \right)  \ide \in \nu (x', x)$, hence the ``twist''
  map
  \begin{eqnarray*}
    \sigma : \car{} & \car{\widehat{X_1}} \longrightarrow &
    \car{\widehat{X_1}}\\
    ((x, x'), m) & \mapsto & \left( (x', x), \left( \sym{X}' \ilam{z}{z m}
    \right)  \ide \right)
  \end{eqnarray*}
  tracked by $\ilam{u \nocomma \nocomma m \nocomma v}{v \sqcap \left( \left(
  \sym{X}' \lam{z}{z m} \right)  \ide \right) \sqcap u}$. But $(\star)$ also
  entails
  \[ \vld{\trans{X}' \assign  \bigM_{x, x', x''} \left( \siex} \nu (x, x')
     \sqcap \siex \nu (x', x'') \rightarrow \siex \nu (x, x'') \right) \]

  Let $x, x', x'' \in \ngh{X}$, $m \in \nu (x, x')$ and $n \in \nu (x', x'')$.
  We then have
  \begin{eqnarray*}
    \left( \trans{X}'  \ilam{z}{z \nocomma m}  \ilam{z}{z n} \right)  \ide &
    \in & \nu (x, x'')
  \end{eqnarray*}
  hence the ``connecting'' map
  \begin{eqnarray*}
    \tau : \car{\widehat{X_1}} \circledast \car{\widehat{X_1}} &
    \longrightarrow & \car{\widehat{X_1}}\\
    (((x, x'), m), ((x', x''), n)) & \mapsto & \left( (x, x''), \left(
    \trans{X}'  \ilam{z}{z \nocomma m}  \ilam{z}{z n} \right)  \ide \right)
  \end{eqnarray*}
  tracked by $\ilam{u \dum  \dum m n v}{u \sqcap \left( \trans{X}'  \ilam{z}{z
  \nocomma m}  \ilam{z}{z n} \right)  \ide \sqcap v}$.
\end{proof}

\label{rem:pseudo}We thus have $\mathcal{K} (\hat{X}) = \left( \ngh{X},
\triplesim \right)$ where
\begin{eqnarray*}
  \eeqq{x}{x'} & = & \left\{\begin{array}{lll}
    c (x) \sqcap c (x') \sqcap \siex \nu (x, x') &  & (x, x') \in \Equ{X}\\
    \bot &  & \text{otherwise}
  \end{array}\right.
\end{eqnarray*}
as $\nu = \nu_{B_X}$ by construction (c.f. Lemma \ref{lem:pseudo}). There are
no ghosts here.

\begin{lemma}
  \label{lem:pseudo-fun}\label{lem:fun}The relation $K : \left( \ngh{X},
  \triplesim \right) \rightarrow (X, \approx)$ given by
  \begin{eqnarray*}
    K : \ngh{X} \times X & \longrightarrow & \mathcal{A}\\
    (x, x') & \mapsto & x \equ x'
  \end{eqnarray*}
  is functional.
\end{lemma}

\begin{proof}
  We have
  \begin{eqnarray*}
    \ext{K} & = & \bigM_{x_1, x_2 \in \ngh{X}} \bigM_{x_1',
    x_2' \in X} \left( \eeq{x_1}{x_1'} \sqcap \eeqq{x_1}{x_2} \sqcap
    \eeq{x_1'}{x_2'} \rightarrow \eeq{x_2}{x_2'} \right)
  \end{eqnarray*}
  Suppose $(x_1, x_2) \nin \Equ{X}$. We then have $\eeqq{x_1}{x_2} = \bot$ \
  so
  \begin{eqnarray*}
    \eeq{x_1}{x_1'} \sqcap \eeqq{x_1}{x_2} \sqcap \eeq{x_1'}{x_2'}
    \rightarrow \eeq{x_2}{x_2'} & = & \top
  \end{eqnarray*}
  in those cases, hence
  \begin{eqnarray*}
    \ext{K} & = & \bigM_{(x_1, x_2) \in \tmop{Equ}^+_X}
    \bigM_{x_1', x_2' \in X} \left( \eeq{x_1}{x_1'} \sqcap
    \eeqq{x_1}{x_2} \sqcap \eeq{x_1'}{x_2'} \rightarrow \eeq{x_2}{x_2'}
    \right)
  \end{eqnarray*}

  But $\ext{K}$ is equivalent to
  \begin{eqnarray*}
    \ext{K}' & \assign & \bigM_{(x_1, x_2) \in \tmop{Equ}^+_X}
    \bigM_{x_1', x_2' \in X} \left( \eeq{x_1}{x_1'} \sqcap c (x_1)
    \sqcap c (x_2) \sqcap \eeq{x_1}{x_2} \sqcap \eeq{x_1'}{x_2'}
    \rightarrow \eeq{x_2}{x_2'} \right)
  \end{eqnarray*}
  by density and furthermore
  \begin{eqnarray*}
    \ilam{u \dum  \dum v w}{\trans{X}  \left( \sym{X} v \right)  \left(
    \trans{X} u v \right)} & \preccurlyeq & \ext{K}'
  \end{eqnarray*}

  Next we have
  \begin{eqnarray*}
    \str{K} & = & \bigM_{x \in \Equ{X}} \bigM_{x' \in X}
    \left( \eeq{x}{x'} \rightarrow | x \triplesim x | \sqcap \eeq{x'}{x'}
    \right)
  \end{eqnarray*}
  But $\str{K}$ is equivalent to
  \begin{eqnarray*}
    \str{K}' & = & \bigM_{x \in \Equ{X}} \bigM_{x' \in X}
    \left( \eeq{x}{x'} \rightarrow c (x) \sqcap \exiplus{X} (x) \sqcap
    \eexi{X} (x') \right)
  \end{eqnarray*}
  by density. Consider the map
  \begin{eqnarray*}
    d : \mathcal{A} & \longrightarrow & \mathcal{A}\\
    a & \mapsto & \left\{\begin{array}{lll}
      c (x) &  & \text{there is } x \in \ngh{X}  \text{such that } a =
      \exiplus{X} (x)\\
      \top &  & \text{otherwise}
    \end{array}\right.
  \end{eqnarray*}
  We have $\lamm{d} \in \sepa$ and furthermore
  \begin{eqnarray*}
    \letin{\varepsilon}{\lam{p}{} \trans{X} p \left( \sym{X} p \right)}
    \text{{\hspace{8em}}} &  & \\
    \letin{\varepsilon'}{\lam{q}{\left( \lamm{d} \right) (\varepsilon q)}}
    \text{{\hspace{7em}}} &  & \\
    \ilam{r}{(\varepsilon' r) \sqcap (\varepsilon r) \sqcap \left( \trans{X}
    \left( \sym{X} p \right) p \right)} & \preccurlyeq & \str{K}'
  \end{eqnarray*}
  Next we obviously have
  \begin{eqnarray*}
    \vld{\sv{K}} & = & \bigM_{x_1 \in \ngh{X}} \bigM_{x_1',
    x_2' \in X} \left( \eeq{x}{x_1'} \sqcap \eeq{x}{x_2'} \rightarrow
    \eeq{x_1'}{x_2'} \right)
  \end{eqnarray*}

  Finally we have
  \begin{eqnarray*}
    \tot{K} & = & \bigM_{x \in \ngh{X}} \left( | x \triplesim x |
    \rightarrow \iex{x' \in X} \left( \eexi{X} (x') \sqcap K (x, x') \right)
    \right)\\
    & = & \bigM_{x \in \ngh{X}} \left( \left( c (x) \sqcap \siex \nu
    (x, x) \right) \rightarrow \iex{x' \in X} \left( \eexi{X} (x') \sqcap
    \eeq{x}{x'} \right) \right)
  \end{eqnarray*}
  and furthermore $\vld{\xi \assign \bigM_{x \in X}} \left( \siex \nu
  (x, x) \rightarrow \exiplus{X} (x) \right)$ by density, hence
  \begin{eqnarray*}
    \ilam{\dum p}{\xi p} & \preccurlyeq & \tot{K}
  \end{eqnarray*}

\end{proof}

\begin{proposition}
  \label{prop:inj}The morphism $[K] : \left( \ngh{X}, \triplesim \right)
  \rightarrow (X, \approx)$ in $\ens [\mathcal{A}]$ is is internally
  injective.
\end{proposition}

\begin{proof}
  Let
  \begin{eqnarray*}
    \mathsf{Inj} & \assign & \bigM_{x_1, x_2 \in \ngh{X}}
    \bigM_{x \in X} \left( K (x_1, x) \sqcap K (x_2, x) \rightarrow
    \eeqq{x_1}{x_2} \right)\\
    & = & \bigM_{x_1, x_2 \in \ngh{X}} \bigM_{x \in X}
    \left( \eeq{x_1}{x} \sqcap \eeq{x_2}{x} \rightarrow c (x_1) \sqcap c (x_2)
    \sqcap \siex \nu (x_1, x_2) \right)
  \end{eqnarray*}

  Consider again the map
  \begin{eqnarray*}
    d : \mathcal{A} & \longrightarrow & \mathcal{A}\\
    a & \mapsto & \left\{\begin{array}{lll}
      c (x) &  & \text{there is } x \in \ngh{X}  \text{such that } a =
      \exiplus{X} (x)\\
      \top &  & \text{otherwise}
    \end{array}\right.
  \end{eqnarray*}
  used in the proof of Lemma \ref{lem:pseudo-fun}. We have in particular
  $\lamm{d} \in \sepa$. Furthermore
  \[ \vld{\chi \assign} \bigM_{x_1, x_2 \in \ngh{X}} \left(
     \eeq{x_1}{x_2} \rightarrow \siex \nu (x_1, x_2) \right) \]
  by density, hence
  %\newpage
  \begin{eqnarray*}
    \letin{j}{\lam{p q}{} \trans{X} p \left( \sym{X} q \right)}
    \text{{\hspace{16em}}} &  & \\
    \letin{e_1}{\lam{p}{} \trans{X}  \left( \sym{X} p \right) p}
    \text{{\hspace{15em}}} &  & \\
    \letin{e_2}{\lam{p}{} \trans{X} p \left( \sym{X} p \right)}
    \text{{\hspace{14em}}} &  & \\
    \letin{\delta}{\lam{p}{\left( \lamm{d} \right) p}} \text{{\hspace{15em}}} &
    & \\
    \ilam{p q}{\letin{\varepsilon}{j p q} ((\delta (e_1 \varepsilon)) \sqcap
    (\delta (e_2 \varepsilon)) \sqcap (\chi \varepsilon)) } & \preccurlyeq &
    \mathsf{Inj}
  \end{eqnarray*}
\end{proof}

Alas, $[K]$ may fail to be internally surjective. For internal surjectivity
means in our case that
\[ \bigM_{x' \in X} \left( \eexi{X} (x') \rightarrow \iex{x \in
   \ngh{X}} K (x, x') \right) \qquad (\star) \]
is in $\sepa$, yet in general we don't have control over what happens on
ghosts. Notice that if $\mathcal{A}$ is compatible with joins, $x'$ being a
ghost means $\eexi{X} (x') = \bot$, hence factors indexed by ghosts evaluate
to $\top$ so we can ``track'' $(\star)$. $[K]$ is then internally surjective
as well, an exemple for this phenomenon being the well-known equivalence
$\pass \simeq \eff$ {\cite{robinson1990colimit,menni2000exact}}. On the other hand if
we co-restrict $\mathcal{K}$ to ``ghostless'' implicative sets, we get an
equivalence of categories.

\begin{theorem}
  Let $\tmmathbf{\tmop{Set}} [\mathcal{A}]^+$ be the full subcategory of
  $\ensa$ with objects implicative sets $\eset{X}$ such that $\eexi{X} (x) \in
  \sepa$ for all $x \in X$. We have
  \begin{eqnarray*}
    \exlex{(\mass)} & \simeq &
    \tmmathbf{\tmop{Set}} [\mathcal{A}]^+
  \end{eqnarray*}
\end{theorem}

\begin{proof}
  The image of the functor $\mathcal{K}:
  \exlex{(\mass)} \rightarrow \ensa$ is contained
  in $\tmmathbf{\tmop{Set}} [\mathcal{A}]^+$. Its co-restriction
  $\mathcal{K}^+ : \exlex{(\mass)} \rightarrow
  \tmmathbf{\tmop{Set}} [\mathcal{A}]^+$ is thus full (c.f. Proposition
  \ref{prop:full}), faithful (c.f. Proposition \ref{prop:K} and internally
  injective (c.f. Proposition \ref{prop:inj}). Let
  \begin{eqnarray*}
    \mathsf{Surj} & \assign & \bigM_{x' \in X} \left( \eexi{X} (x')
    \rightarrow \iex{x \in \ngh{X}} K (x, x') \right)
  \end{eqnarray*}
  Suppose $\eexi{X} (x')$. We then have $x' \in \ngh{X}$ so in particular $K
  (x', x') = \eexi{X} (x')$, thus $\tmop{id} \preccurlyeq \mathsf{Surj}$.
  Hence $[K]$ is iso, thus $\mathcal{K}^+$ is an equivalence of categories.
\end{proof}

Wrapping it up, the following picture begins to emerge

\begin{center}
  \raisebox{-0.5\height}{\includegraphics[width=14.8909550045914cm,height=2.77676111767021cm]{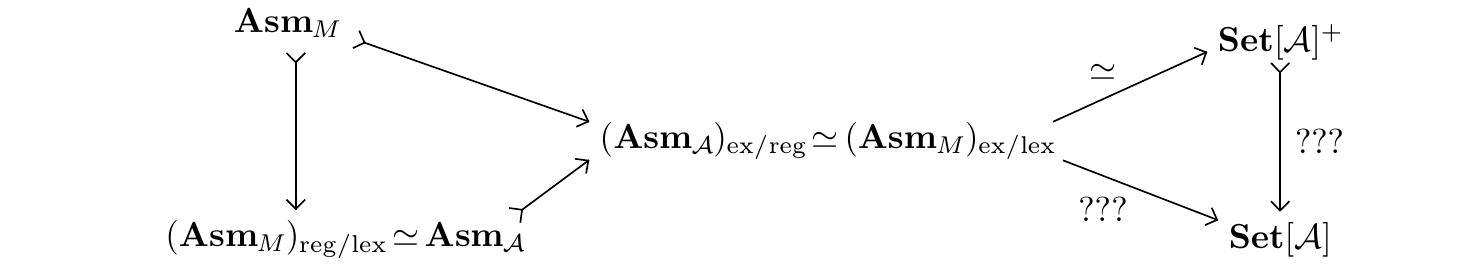}}
\end{center}

All the categories involved but $\mass$ and
$\asma$ are topoi. In particular, we
see that $\asma$ is a subcategory of
$\ensa$ yet the embedding is not the trivial one. As a matter of fact,
assemblies embed as kernel pairs of their projective covers. A
characterisation of those implicative algebras for which we have
$\exlex{(\mass)} \simeq \ensa$ is an open question.
We may have more to say about this circle of ideas in subsequent elaborations.


\begin{thebibliography}{10}
  \bibitem[1]{barr1971exact}Michael Barr, Pierre Grillet, Donovan Van Osdol,
  and  Michael Barr. {\newblock}\tmtextit{Exact categories}.
  {\newblock}Springer, 1971.{\newblock}

  \bibitem[2]{carboni1995some}Aurelio Carboni. {\newblock}Some free
  constructions in realizability and proof theory.
  {\newblock}\tmtextit{Journal of pure and applied algebra}, 103(2):117--148,
  1995.{\newblock}

  \bibitem[3]{carboni1982free}Aurelio Carboni and Celia Magno.
  {\newblock}The free exact category on a left exact one.
  {\newblock}\tmtextit{Journal of the Australian Mathematical Society},
  33(3):295--301, 1982.{\newblock}

  \bibitem[4]{Castro:2023aa}F{\'e}lix Castro, Alexandre Miquel, and  Krzysztof
  Worytkiewicz. {\newblock}Implicative assemblies. {\newblock}Submitted,
  2023.{\newblock}

  \bibitem[5]{hyland1982effective}Martin Hyland. {\newblock}The effective
  topos. {\newblock}In \tmtextit{Studies in Logic and the Foundations of
  Mathematics},  volume  110,  pages  165--216. Elsevier, 1982.{\newblock}

  \bibitem[6]{hyland1980tripos}Martin Hyland, Peter Johnstone, and  Andrew
  Pitts. {\newblock}Tripos theory. {\newblock}In \tmtextit{Mathematical
  Proceedings of the Cambridge philosophical society},  volume~88,  pages
  205--232. Cambridge University Press, 1980.{\newblock}

  \bibitem[7]{kinoshita2014category}Yoshiki Kinoshita  and  John Power.
  {\newblock}Category theoretic structure of setoids.
  {\newblock}\tmtextit{Theoretical Computer Science}, 546:145--163,
  2014.{\newblock}

  \bibitem[8]{maietti2015unifying}Maria~Emilia Maietti  and  Giuseppe
  Rosolini. {\newblock}Unifying exact completions.
  {\newblock}\tmtextit{Applied Categorical Structures}, 23:43--52,
  2015.{\newblock}

  \bibitem[9]{menni2000exact}Mat{\'i}as Menni. {\newblock}\tmtextit{Exact
  completions and toposes}. {\newblock}PhD thesis, 2000.{\newblock}

  \bibitem[10]{miquel2020implicative}Alexandre Miquel. {\newblock}Implicative
  algebras: a new foundation for realizability and forcing.
  {\newblock}\tmtextit{Mathematical Structures in Computer Science},
  30(5):458--510, 2020.{\newblock}

  \bibitem[12]{robinson1990colimit}Edmund Robinson  and  Giuseppe Rosolini.
  {\newblock}Colimit completions and the effective topos.
  {\newblock}\tmtextit{The Journal of Symbolic Logic}, 55(2):678--699,
  1990.{\newblock}

  \bibitem[13]{rosolini2019elementary}Giuseppe Rosolini, Fabio Pasquali, and
  Maria~Emilia Maietti. {\newblock}Elementary quotient completions, church's
  thesis, and partioned assemblies. {\newblock}\tmtextit{Logical Methods in
  Computer Science}, 15, 2019.{\newblock}

  \bibitem[14]{shulman2021derivator}Michael Shulman. {\newblock}The derivator
  of setoids. {\newblock}\tmtextit{Cahiers de topologie et g{\'e}om{\'e}trie
  diff{\'e}rentielle cat{\'e}goriques}, 2023.{\newblock}
\end{thebibliography}
\end{document}